\documentclass[12pt]{article}
\usepackage[utf8]{inputenc}
\usepackage{amsthm}
\usepackage{amsmath}
\usepackage{amssymb}
\usepackage{amscd}
\usepackage[all]{xy}
\usepackage[margin=1in]{geometry}
\usepackage{comment}
\usepackage{inputenc}
\usepackage{pdfpages}
\usepackage{tikz}
\usepackage{pgfplots}
\usepackage{latexsym}
\usetikzlibrary{patterns}

\begin{document}
\tikzset{mark size=1.5}

\title{The Ap\'{e}ry Set of a Good Semigroup}
\author{Marco D'Anna, Lorenzo Guerrieri, Vincenzo Micale 
\thanks{Universit\`a di Catania, Dipartimento di Matematica e Informatica, Viale A. Doria, 6, 95125 Catania, Italy \textbf{Email addresses:} mdanna@dmi.unict.it, guelor@guelan.com, vmicale@dmi.unict.it }
}

\maketitle


\begin{abstract}

\noindent We study the Ap\'ery Set of good subsemigoups of $\mathbb N^2$, a class of semigroups containing the value semigroups 
of curve singularities with two branches. Even if this set in infinite, we show that, for the Ap\'ery Set of such semigroups, we can define a partition in "levels"
that allows to generalize many properties of the Ap\'ery Set of numerical semigroups, i.e. value semigroups of one-branch singularities.
\medskip

\end{abstract}


\newtheorem{theorem}{Theorem}[section]
\newtheorem{lemma}[theorem]{Lemma}
\newtheorem{prop}[theorem]{Proposition}
\newtheorem{corollary}[theorem]{Corollary}
\newtheorem{problem}[theorem]{Problem}
\newtheorem{construction}[theorem]{Construction}

\theoremstyle{definition}
\newtheorem{defi}[theorem]{Definitions}
\newtheorem{definition}[theorem]{Definition}
\newtheorem{remark}[theorem]{Remark}
\newtheorem{example}[theorem]{Example}
\newtheorem{question}[theorem]{Question}
\newtheorem{comments}[theorem]{Comments}

\newtheorem{discussion}[theorem]{Discussion}

\newcommand{\N}{\mathbb{N}}
\newcommand{\Z}{\mathbb{Z}}
\newcommand{\al}{\boldsymbol{\alpha}}
\newcommand{\be}{\boldsymbol{\beta}}
\newcommand{\de}{\boldsymbol{\delta}}
\newcommand{\e}{\boldsymbol{e}}
\newcommand{\om}{\boldsymbol{\omega}}
\newcommand{\g}{\boldsymbol{\gamma}}
\newcommand{\te}{\boldsymbol{\theta}}
\newcommand{\he}{\boldsymbol{\eta}}
\newcommand{\C}{\boldsymbol{c}}
\newcommand{\K}{\mathrm{K}}
\newcommand{\Ap}{\mathrm{\bf Ap}}
\def\min{\mbox{\rm min}}
\def\max{\mbox{\rm max}}
\def\ff{\frak}
\def\Spec{\mbox{\rm Spec }}
\def\Proj{\mbox{\rm Proj }}
\def\hgt{\mbox{\rm ht }}
\def\type{\mbox{ type}}
\def\Hom{\mbox{ Hom}}
\def\rank{\mbox{ rank}}
\def\Ext{\mbox{ Ext}}
\def\Ker{\mbox{ Ker}}
\def\Max{\mbox{\rm Max}}
\def\End{\mbox{\rm End}}
\def\xpd{\mbox{\rm xpd}}
\def\Ass{\mbox{\rm Ass}}
\def\emdim{\mbox{\rm emdim}}
\def\epd{\mbox{\rm epd}}
\def\repd{\mbox{\rm rpd}}
\def\ord{\mbox{\rm ord}}
\def\Tr{\mbox{\rm Tr}}
\def\Res{\mbox{\rm Res}}
\def\sdefect{\mbox{\rm sdefect}}
\maketitle

\section{Introduction}

The concept of  \emph{good semigroup} was formally defined in \cite{a-u} in order to study value semigroups of Noetherian analytically unramified 
one-dimensional semilocal reduced rings, e.g. the local rings arising from curve singularities (and from their blowups), possibly with more than one branch; the properties of these semigroups were already considered in \cite{two, c-d-gz, C-D-K, danna1, felix, delgado, garcia}, but it was in \cite{a-u} that their structure was systematically studied.
Similarly to the one branch case, when the value semigroup is a numerical semigroup, the properties of the rings can be translated and studied at
semigroup level. For example, the celebrated result by Kunz (see \cite{kunz}) that a one-dimensional analytically irreducible local domain is Gorenstein if and only 
if its value semigroup is symmetric can be generalized to analytically unramified rings (see \cite{delgado} and also \cite{C-D-K}); in the same way the numerical characterization of the canonical module in the analytically irreducible case (see \cite{jager}) can be given also in the more general case (see \cite{danna1}).

However good semigroups present some problems that make difficult their study; first of all they are not finitely generated as monoid (even if they
can be completely determined by a finite set of elements (see \cite{garcia}, \cite{C-H} and \cite{D-G-M-T}) and they are not closed under finite intersections. 
Secondly, the behavior of the good ideals of good semigroups (e.g. the ideals
arising as values of ideals of the corresponding ring) is not good at all, in the sense that the class of good ideals is not closed under 
sums and differences (see e.g. \cite{a-u} and \cite{KST}).

Hence, unlike what happens for numerical semigroups (in analogy to analytically irreducible domains), it is not clear how to define the concept of complete intersection 
good semigroups and also the concepts of embedding dimension and type for these semigroups.

Moreover, in the same paper \cite{a-u}, it is shown that the class of good semigroups is larger than the class of value semigroups and, at the moment, 
no characterization of value good semigroups is known (while, for the numerical semigroup case, it is easily seen that any such semigroup is the value 
semigroup of the ring of the corresponding monomial curve).
This means that to prove a property for good semigroups it is not possible to take advantage of the 
nature of value semigroups and it is necessary to work only at semigroup level.

Despite this bad facts, there is a concept quite natural to define, that seems very promising in order to study good semigroups and
to translate at numerical level other ring concepts: the \emph{Ap\'ery Set}.
In general, given any monoid $S$ and any element $s \in S$, the $\Ap(S,s)$ is defined as the set
$\{t \in S : \  t-s \notin S\}$ (where the $-$ is taken in the group generated by $S$).
For studying numerical semigroups, this concept reveals to be very useful and it is also a bridge 
between semigroup and ring properties, since many important ring properties are stable
under quotients with respect to principal ideals generated by a nonzero divisor $(x)$ and 
the values of the nonzero elements in $R/(x)$ are exactly the elements of $\Ap(S,v(x))$. 
This strategy was used, e.g., in \cite{bryant}, \cite{co-ja-za}, \cite{da-mi-sa}, \cite{da-ja-st} 
taking $(x)$ to be a minimal reduction of the maximal ideal
that, in this situation, correspond to an element of minimal nonzero value.

For good semigroups, the notion of $\Ap(S,s)$ was used in \cite{apery}, in order to obtain an algorithmic 
characterization of those good semigroups that are value semigroups of a plane singularity with two branches.
In that paper, using deeply the structure of the rings associated to plane singularities, it is proved that 
$\Ap(S, \boldsymbol{e}$) (where $\boldsymbol{e}=(e_1,e_2)$ is the minimal nonzero element in $S \subset \mathbb N^2$)
can be divided in $e=e_1+e_2$ subsets, where the integer $e$ corresponds to the multiplicity of the ring.

In this paper we want to investigate the Ap\'ery Set of a good semigroup. 
Again the problem is that we have to deal
with an infinite set; so, first of all, we want to understand if there is a natural partition of it in $e$ subsets, where $e$ is the sum of the components of 
the minimal nonzero element of $S$ and, in case $S$ is a value semigroup, it represents also the multiplicity of the corresponding ring; 
to answer to this question we decided to restrict to the good subsemigroups of $\mathbb N^2$,
otherwise the technicalities would increase too much.
After finding a possible partition $\Ap(S)=\bigcup_{i=1}^e A_{i}$, we prove that, if $S$ is the value semigroup of a ring $(R, \mathfrak m, k)$, it is possible to 
choose $e$ elements 
$\boldsymbol{\alpha}_i$ in the Ap\'ery Set, one for each $A_i$, so that, taking any element $f_i \in R$ of valuation $v(f_i)=\boldsymbol{\alpha}_i$,
the classes $\bar f_i$ are a basis of the $e$-dimensional $k$-vector space $R/(x)$ (where $x$ is a minimal reduction of $\mathfrak m$).
This fact make us confident that the definition of the partition is the one we where looking for.

At this point it is natural to investigate if it is possible to generalize the well known characterization of symmetric numerical semigroups given via their Ap\'ery Set.
It turns out that also good symmetric semigroups have $\Ap(S,\boldsymbol{e}$) whose partition satisfies a duality property
similar to the duality that holds for the numerical case.

The structure of the paper is the following: after recalling in Section 2 all the preliminary definitions and results needed for the rest of the paper, in Section 3
we define the partition of the Ap\'ery Set of $S$ and we prove that this partition produces exactly $e_1+e_2$ levels (Theorem \ref{livelli}); successively we deepen the 
study of the structure of Ap\'ery Set (Theorem \ref{infiniti}) and we prove that, in the case of value semigroups, the partition allows to find a basis 
for $R/(x)$ as explained above (Theorem \ref{anelloquoziente}).

In Section 4 we study the properties of the Ap\'ery Set of symmetric good semigroups with particular attention to duality properties of its elements (Proposition \ref{simmetria}
and Theorem \ref{proiezione}) and in Section 5 we use these results to prove a duality for the levels, characterizing the symmetric semigroups, in analogy to the duality of Ap\'ery Set in the numerical semigroup case
(Theorem \ref{duality}). Finally we deepen this duality showing that we can find sequences of elements, one for each level, that have the same duality properties 
of the Ap\'ery Set of the numerical semigroups (Theorem \ref{sequenza2}).

\section{Preliminaries}

Let $\mathbb N$ be the set of nonnegative integers. As usual, $\le$ stands for the natural partial ordering in $\mathbb N^2$: set $\al=(\alpha_1, \alpha_2), \be=(\beta_1, \beta_2)$, then $\boldsymbol{\alpha}\le \boldsymbol{\beta}$ if $\alpha_1\le \beta_1$ and $\alpha_2\le \beta_2$. Trough this paper, if not differently specified, when referring to minimal or maximal elements of a subset of $\mathbb N^2$, we refer to minimal or maximal elements with respect to $\le$.
Given $\boldsymbol{\alpha},\boldsymbol{\beta}\in \mathbb N^2$, the infimum of the set $\{\boldsymbol{\alpha},\boldsymbol{\beta}\}$ (with respect to $\le$) will be denoted by $\boldsymbol{\alpha}\wedge \boldsymbol{\beta}$. Hence $$\boldsymbol{\alpha}\wedge \boldsymbol{\beta}=(\min(\alpha_1,\beta_1),\min(\alpha_2,\beta_2)).$$

Let $S$ be a submonoid of $(\mathbb N^2,+)$. We say that $S$ is a \emph{good semigroup} if

\begin{itemize}
	\item[(G1)] for all $\boldsymbol{\alpha},\boldsymbol{\beta}\in S$, $\boldsymbol{\alpha}\wedge \boldsymbol{\beta}\in S$;
	\item[(G2)] if $\boldsymbol{\alpha},\boldsymbol{\beta}\in S$ and $\alpha_i=\beta_i$ for some $i\in \{1,2\}$, then there exists $\boldsymbol{\delta}\in S$ such that $\delta_i>\alpha_i=\beta_i$, $\delta_j\ge\min\{\alpha_j,\beta_j \}$ for $j\in\{1,2\}\setminus\{i\}$ and $\delta_j=\min\{\alpha_j,\beta_j \}$ if $\alpha_j\ne \beta_j$;
	\item[(G3)] there exists $\boldsymbol{c}\in S$ such that $\boldsymbol{c}+\mathbb N^2\subseteq S$.
\end{itemize}

A good subsemigroup of $\mathbb N^2$ is said to be \emph{local} if $\boldsymbol{0}=(0,0)$ is its
only element with a zero component. In the following we will work only with local good semigroups 
hence we will omit the word local.

Notice that, from condition (G1), if $\boldsymbol{c}$ and $\boldsymbol{d}$ fulfill (G3), then so does $\boldsymbol{c}\wedge \boldsymbol{d}$. So there exists a minimum $\boldsymbol{c}\in \mathbb N^2$ for which condition (G3) holds. Therefore we will say that
\[
\boldsymbol{c}:=\min\{\boldsymbol{\alpha}\in \mathbb Z^2\mid \boldsymbol{\alpha}+\mathbb N^2\subseteq S\}
\]
is the \emph{conductor} of $S$.
We denote $\boldsymbol{\gamma}:=\boldsymbol{c}-\textbf{1}$.

In light of \cite[Proposition 2.1]{a-u}, value semigroups of Noetherian, analytically unramified, residually rational, one-dimensional, reduced semilocal rings with two minimal primes are good subsemigroups of $\mathbb N^2$ and $R$ is local if and only it its value semigroup is local; in the rest of this paper we will always assume these hypotheses on the rings $R$ unless differently stated. 

We give the following technical definitions that are commonly used in the literature about good semigroups:
\begin{itemize}
	\item[(1)] $\Delta_i (\boldsymbol{\alpha}):=\{\boldsymbol{\beta}\in \mathbb Z^2 \mid \alpha_i=\beta_i \textup{ and } \alpha_j<\beta_j \textup{ for } j\ne i\}$,
	\item[(2)] $\Delta_i^S (\boldsymbol{\alpha}) := \Delta_i (\boldsymbol{\alpha}) \cap S$,
	\item[(3)] $\Delta(\boldsymbol{\alpha}):=\Delta_1 (\boldsymbol{\alpha})\cup\Delta_2 (\boldsymbol{\alpha})$,
	\item[(4)] $\Delta^S (\boldsymbol{\alpha}) := \Delta (\boldsymbol{\alpha}) \cap S$.
\end{itemize}

An element $\al \in S$ is said to be \it absolute \rm if $ \Delta^S(\al) = \emptyset$.
By definition of conductor we immediately get $ \Delta^S(\g) = \emptyset$. Given $\boldsymbol{\alpha}, \boldsymbol{\beta}\in\mathbb N^2$, we say that $\boldsymbol{\beta}$ is \emph{above} $\boldsymbol{\alpha}$ if $\boldsymbol{\beta}\in\Delta_1 (\boldsymbol{\alpha})$ and that $\boldsymbol{\beta}$ is \emph{on the right} of $\boldsymbol{\alpha}$ if $\boldsymbol{\beta}\in\Delta_2 (\boldsymbol{\alpha})$. 

\begin{remark}
	Let $\boldsymbol{c}=(c_1,c_2)$ be the conductor of $S$. By properties (G1) and (G2), we have that if $\boldsymbol{\alpha}=(\alpha_1,c_2)\in S$, then $\Delta_1 (\boldsymbol{\alpha})=\Delta_1^S (\boldsymbol{\alpha})$ (that is, each point $\boldsymbol{\beta}\in\mathbb N^2$, above $\boldsymbol{\alpha}$, is in $S$). Similarly, if $\boldsymbol{\alpha}=(c_1,\alpha_2)\in S$, then $\Delta_2 (\boldsymbol{\alpha})=\Delta_2^S (\boldsymbol{\alpha})$ (that is, each point $\boldsymbol{\beta}\in\mathbb N^2$, on the right of $\boldsymbol{\alpha}$, is in $S$).
\end{remark}

\bigskip

The Ap\'{e}ry Set of $S$ with respect to $\boldsymbol{\beta} \in S$ is defined as the set 
$$\Ap(S, \be)=\left\{\boldsymbol{\alpha}\in S| \boldsymbol{\alpha}-\be \notin S\right\}.$$
Property (G1) implies that for a local good semigroup there exists a smallest non zero element
that we will denote by $\boldsymbol{e}=(e_1,e_2)$. 
We will usually consider the Ap\'{e}ry Set of $S$ with respect to $\boldsymbol{e}$ and, in this case, we will simply write
$\Ap(S)$.

\begin{figure}[h]  
  
  \begin{tikzpicture}[scale=1.2]
	\begin{axis}[grid=major, ytick={0,3,4,5,6,7,8, 9, 10, 11}, xtick={0,2,3,4,5,6}, yticklabel style={font=\tiny}, xticklabel style={font=\tiny}]
\addplot[only marks] coordinates{ (2,3) (4,6) (4,7) (5,6) (5,9) (5,10) (5,11) (5,12) (5,13) (5,14)  
(6,6)   (7,6)   (8,6) (6,9) (6,10)(6,11)  
(7,9) (7,10) (8,9) (8,10) (7,12) (7,13) (7,14) (8,12) (8,13) (8,14)  }; 
 \addplot[only marks, mark=o] coordinates{ (0,0) (2,4) (3,3) 
  (3,6) (3,7) (3,8) (4,3) (5,3) (6,3) (7,3) (8,3)  
 (5,7) (6,7) (7,7) (8,7) (4,8) (3,9) (3,10) (3,11) (3,12) (3,13) (3,14)  
 (6,12) (6,13) (6,14) (7,11) (8,11)  }; 
 \end{axis}
  \end{tikzpicture}
  \caption{ \scriptsize The value semigroup of $ k[[X,Y,Z]]/ (X^3-Z^2) \cap (X^3-Y^4) $. The elements of the Ap\'{e}ry Set are indicated with $\circ$.}
\label{semtalk}
\end{figure}
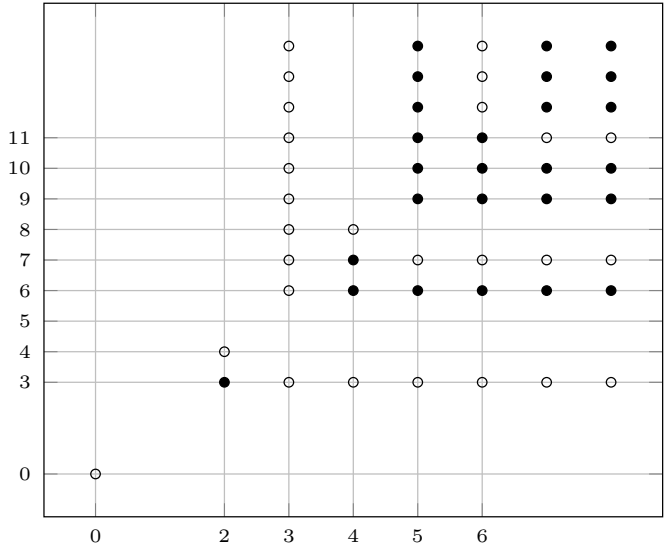

\begin{remark}
	By definitions of conductor of $S$ and of $\Ap(S)$, we have 
	$$
	\{\boldsymbol{\alpha}\in\mathbb N^2\ |\ \boldsymbol{\alpha}\ge\boldsymbol{\gamma}+\boldsymbol{e}+\boldsymbol{1}\}\cap\Ap(S)=\emptyset.
	$$
\end{remark}

Let $A$ be a subset of $\mathbb N^2$ and let $\boldsymbol{\alpha}, \boldsymbol{\beta}\in A$, we say that $\boldsymbol{\beta}$ is a consecutive point of $\boldsymbol{\alpha}$ in $A$ if, for every $\boldsymbol{\mu}\in\mathbb N^2$ with $\boldsymbol{\alpha}<\boldsymbol{\mu}<\boldsymbol{\beta}$, we necessary have $\boldsymbol{\mu}\notin A$. 

\begin{remark} Let $\al, \be \in \Ap(S)$.
	If $\boldsymbol{\beta}$ is a consecutive point of $\boldsymbol{\alpha}$ in $S$, then $\boldsymbol{\beta}$ is a consecutive point of $\boldsymbol{\alpha}$ in $\Ap(S)$. The converse it is not true.
\end{remark}

Let $\boldsymbol{\alpha}, \boldsymbol{\beta}\in A$. The chain of points in $A$, $\{\boldsymbol{\alpha}=\boldsymbol{\alpha}_1<\cdots<\boldsymbol{\alpha}_h<\cdots<\boldsymbol{\alpha}_n=\boldsymbol{\beta}\}$, with $\boldsymbol{\alpha}_{h+1}$ consecutive of $\boldsymbol{\alpha}_h$, is called \emph{saturated} of length $n$. In this case $\boldsymbol{\alpha}$ and $\boldsymbol{\beta}$ are called the initial point and the final point, respectively, of the chain.

A \emph{relative ideal} of a good semigroup $S$ is a subset $\emptyset\ne E \subseteq \mathbb Z^2$ such that $E + S \subseteq E$ and $\boldsymbol{\alpha}+E\subseteq S$ for some $\boldsymbol{\alpha}\in S$. A relative ideal $E$ contained in $S$ is simply called an \it ideal\rm.
If $E$ satisfies (G1) and (G2), then we say that $E$ is a \emph{good ideal} of $S$ (condition (G3) follows from the definition of good relative ideal).

\begin{prop} \rm (\cite{danna1} and \cite{KST}) \it
	All the saturated chains in a good ideal $E$ of $S$ with fixed initial and final points have the same length. 
\end{prop}

If $\boldsymbol{\alpha},\boldsymbol{\beta}\in E$, we denote by
$d_E(\boldsymbol{\alpha},\boldsymbol{\beta})$ the common length of all the saturated chains in the good ideal
$E$ with initial point $\al$ and final point $\be$. 
Moreover, if $E \supseteq F$ are two good ideals, considering 
$\boldsymbol m_E$ and $\boldsymbol m_F$ the minimal elements, respectively, of $E$ and $F$ and taking $\al \geq \boldsymbol{c}_F$ where $\boldsymbol{c}_F$ is the 
conductor of $F$, it is possible to define the following distance function $d(E\setminus F)=d_E(\boldsymbol m_E, \al)-d_F(\boldsymbol m_F, \al)$
(cf \cite{danna1} and \cite{KST} to see that it is a well defined distance). This distance funtion allows to translate many ring properties at semigroup level,
since, if $I \supseteq J$ are two fractional ideals of $R$, it is proved, in the same papers, that the length $\lambda_R(I/J)$ equals 
$d_{v(R)}(v(I) \setminus v(J))$.

\begin{remark}
\label{remnumerical}
In the numerical semigroup case, that is $S=v(R)$ with $(R,\mathfrak{m})$ a one-dimensional, Noetherian, analitically irreducible, residually rational, local domain, it is well known that, if we denote by $x$ a minimal reduction of $\mathfrak{m}$ (i.e. an element of minimal value $e$), then $\overline{y}\ne 0$ in $R/(x)$ if and only if $v(y)\in\Ap(S)=\left\{\omega_0,\dots,\omega_{e-1}\right\}$. Moreover, if $\omega_i=v(y_i)$, then $\left\{\overline{y}_0,\dots \overline{y}_{e-1}\right\}$ are linear independent in the $R/\mathfrak{m}$-vector space $R/(x)$ and so they form a basis for it. 
\end{remark}

The first part of this remark can be easily generalized as follows, while we will be able to generalize the second part using all the results of the next section 
(see Theorem \ref{anelloquoziente}).

\begin{prop}
\label{prop1}
 Let $y \in R$; then
	$\overline{y}\ne 0$ in $R/(x)$ if and only if $v(y)\in\Ap(S)$.
\end{prop}
\begin{proof}
	By definition $\overline{y}\ne 0$ in $R/(x)$ if and only if $y \notin (x)$ that is equivalent to say that $yx^{-1}\notin R$. Since $v(yx^{-1})=v(y)-v(x)$,
	if $v(y)\in\Ap(S)$ we immediately get that 	$yx^{-1}\notin R$, i.e. $\overline{y}\ne 0$ in $R/(x)$. 
	Conversely, assume that $v(y)\notin\Ap(S)$, i.e. $v(yx^{-1})=v(y)-v(x)=v(r)$, for some $r \in R$. Since $R$ and both its projections on the two minimal primes are residually rational, it follows that there exists an invertible $u$ in $R$ such that $v(yx^{-1}-ur) > v(r)$; moreover we can choose $u$ in order to increase the first or the second component, as we prefer.
	Hence, applying repeatedly this argument we obtain, after a finite number of steps, that $v(yx^{-1}-u'r') \geq c$, that implies $yx^{-1}-u'r' \in (R:\overline R) \subset R$;
	therefore $y \in (x)$, a contradiction.	
\end{proof}

\begin{remark}
\label{multiplycity}
When $S$ is the value semigroup of a ring $(R, \mathfrak m)$, it is not difficult to see that
an element $x$ is a minimal reduction of $\mathfrak{m}$ if and only if  
$v(x)=\boldsymbol{e}$; hence the integer $e=e_1+e_2$ coincides with the multiplicity of $R$: in fact  $e(R)=\lambda_{R}(R/(x))=\dim_{R/\mathfrak{m}}(R/(x))$ (where $\lambda_R$ denotes the length of an $R$ module); now, using the computation of lengths
	explained above, it is not difficult to check that 
	$\lambda_{R}(R/(x))=d(S\setminus \boldsymbol e +S)=e_1+e_2$.
\end{remark}

It is useful to remark, at this point, that  $S\setminus \boldsymbol e +S = \Ap(S)$. Hence, our first goal is, 
starting form a good semigroup $S\subseteq\mathbb N^2$, to get a partition $\Ap(S)$ in $e_1+e_2=e$ levels that should correspond to the $e$ elements of the Ap\'ery Set of a numerical semigroup.
After that we would like to find $e$ elements $\boldsymbol{\alpha}_1,\dots,\boldsymbol{\alpha}_{e}$ of the Ap\'{e}ry set of $S$, one in each class of the partition,
with the property that, if $S=v(R)$ and $f_i\in R$ are such that $v(f_i)=\boldsymbol{\alpha}_i$, then $\overline{f}_1,\dots, \overline{f}_{e}$ are linear independent in the 
$R/\mathfrak{m}$-vector space 
$R/(x)$ and so they form a basis for it.

We notice again that, given a good semigroups, it is not known a procedure to see if it is a value semigroup of a ring or not; 
so we are forced to use semigroup arguments without the possible help of ring techniques.

\medskip

As in the numerical semigroup case, a symmetric good semigroup has a duality property. Indeed in \cite{delgado}, 
a good semigroup $S$ is said to be \emph{symmetric} if
$$
\boldsymbol{\alpha}\in S\ \Leftrightarrow \ \mathrm \Delta^S(\boldsymbol{\gamma} -\boldsymbol{\alpha})=\emptyset.
$$

Moreover, in the numerical semigroup case, the symmetry of the semigroup $S$ can be characterized by a symmetry of its the Ap\'ery Set: if we order its elements in increasing order
$\Ap(S)=\{w_1, \dots w_e\}$, then $S$ is symmetric if and only if $w_i+w_{e-i+1}=w_e$ for every $i=1, \ldots, e$.

In this paper we also look for an analogue property for $\Ap(S)$ when $S$ is a symmetric good subsemigroup of $\mathbb N^2$.

%
%

\section{Ap\'ery Set of good semigroups in $\mathbb N^2$}

In order to get the 
partition of the Ap\'ery Set of $S$ we are looking for, we need to introduce 
a new relation on $\mathbb N^2$ (as it is done in \cite{a-u}): 
we say that $(\alpha_1,\alpha_2)\le\le(\beta_1,\beta_2)$ if an only if $(\alpha_1,\alpha_2)=(\beta_1,\beta_2)$ or $(\alpha_1,\alpha_2)\ne(\beta_1,\beta_2)$ and $(\alpha_1,\alpha_2)\ll(\beta_1,\beta_2)$, where the last means $\alpha_1<\beta_1$ and $\alpha_2<\beta_2$.


We define the following subsets of $\Ap(S)$:
$$ B^{(1)}=\{\al \in \Ap(S) : \al \  \mbox{is maximal with respect to} \le\le\},$$
$$ C^{(1)}:= \{ \al \in B^{(1)} : \al= \be_1 \wedge \be_2 \mbox{ for some } \be_1, \be_2 \in B^{(1)} \setminus \lbrace \al \rbrace \}
\ \ \mbox{and} \ \ D^{(1)}=B^{(1)}\setminus C^{(1)}.$$\\
Assume $i>1$ and that $D^{(1)},\dots , D^{(i-1)}$ have been defined; we set
$$ B^{(i)}=\{\al \in \Ap(S)\setminus (\bigcup_{j < i} D^{(j)}) : \al \  \mbox{is maximal with respect to} \le\le\},$$
$$ C^{(i)}:= \{ \al \in B^{(i)} : \al= \be_1 \wedge \be_2 \mbox{ for some } \be_1, \be_2 \in B^{(i)} \setminus \lbrace \al \rbrace \}\ \ \mbox{and} \ \ D^{(i)}=B^{(i)}\setminus C^{(i)}.$$\\
Clearly, for some $N \in \mathbb N_+$, we have $\Ap(S)=\bigcup_{i=1}^N D^{(i)}$  and 
 $D^{(i)}\cap D^{(j)}=\emptyset$, for any $i\neq j$. 
 For simplicity we prefer to number the set of the partition in increasing order,
 so we set $A_i=D^{(N+1-i)}$. Hence 
 $$\Ap(S)=\bigcup_{i=1}^N A_i $$
 We want to prove that $N=e_1+e_2$. We will call the sets $A_i$ \it levels \rm of the Ap\'ery Set.

 \begin{figure}[h] 
  
  \begin{tikzpicture}[scale=1.2]
	\begin{axis}[grid=major, ytick={0,3,4,5,6,7,8, 9, 10, 11}, xtick={0,2,3,4,5,6}, yticklabel style={font=\tiny}, xticklabel style={font=\tiny}]
\addplot[only marks] coordinates{ (2,3) (4,6) (4,7) (5,6) (5,9) (5,10) (5,11) (5,12) (5,13) (5,14)  
(6,6)   (7,6)   (8,6) (6,9) (6,10)(6,11)  
(7,9) (7,10) (8,9) (8,10) (7,12) (7,13) (7,14) (8,12) (8,13) (8,14)  }; 
 \addplot[only marks, mark=text, mark options={scale=1,text mark={ \tiny 1}}, text mark as node=true] coordinates{(0,0)}; 
 \addplot[only marks, mark=text, mark options={scale=1,text mark={ \tiny 2}}, text mark as node=true] coordinates{ (2,4) (3,3)}; 
 \addplot[only marks, mark=text, mark options={scale=1,text mark={ \tiny 3}}, text mark as node=true] coordinates{ (3,6) (3,7) (3,8) (4,3) (5,3) (6,3) (7,3) (8,3) }; 
 \addplot[only marks, mark=text, mark options={scale=1,text mark={ \tiny 4}}, text mark as node=true] coordinates{(5,7) (6,7) (7,7) (8,7) (4,8) (3,9) (3,10) (3,11) (3,12) (3,13) (3,14) }; 
 \addplot[only marks, mark=text, mark options={scale=1,text mark={ \tiny 5}}, text mark as node=true] coordinates{(6,12) (6,13) (6,14) (7,11) (8,11)  }; 
 \end{axis}
  \end{tikzpicture}
  \caption{\scriptsize The Ap\'{e}ry Set of the semigroup in Figure \ref{semtalk}. We mark the elements of the set $A_i$ with the number $i.$  
   }
\label{semtalklev}
\end{figure}
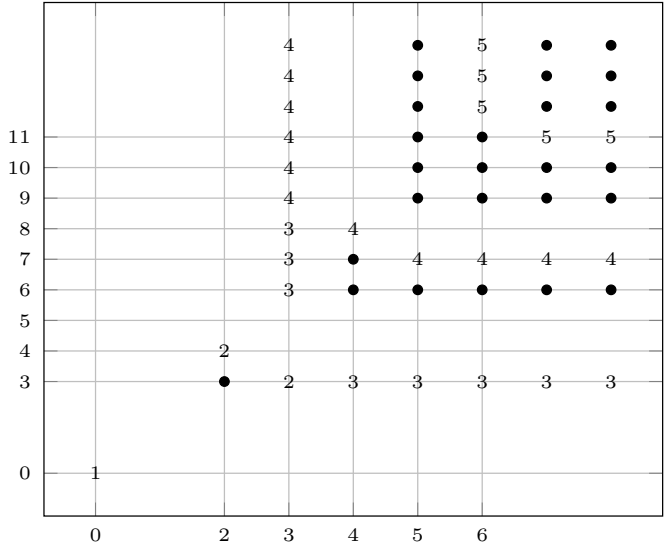

\begin{remark} \label{rem1}
	It is straightforward to see that $A_N=\Delta(\g+\boldsymbol{e})=\Delta^S(\g+\boldsymbol{e})$
	and that $A_1=\{\boldsymbol {0}\}$. Moreover, if $\al, \be \in \Ap(S)$, $\al \ll \be$ and $\al \in A_i$, then $\be \in A_j$ for some $j>i$.
\end{remark}

\begin{lemma}\label{basic-Ap} The sets $A_i$ satisfy the following properties:
	\begin{itemize}
		\item[(1)] for any $\al \in A_i$ there exists $\be \in A_{i+1}$ such that 
		$\al \ll \be$ or $\al = \be_1 \wedge \be_2$ with $\be_1, \be_2 \in A_{i+1}$ (both cases can
		happen at the same time); 
		\item[(2)] for every $\al \in A_i$ and $\be \in A_j$, with $j\geq i$, $\be \not\ll \al$; 
		\item[(3)] if $\al \in A_i$, $\be \in \Ap(S)$ 
		and $\be \geq \al$, then $\be \in A_i \cup \dots \cup A_N$;
		\item[(4)] if $\al=(\alpha_1,\alpha_2), \be=(\alpha_1,\beta_2) \in A_i$, with $\alpha_2 < \beta_2$, then
		for any $\de=(\delta_1,\delta_2) \in \Ap(S)$ such that $\delta_1 > \alpha_1$ and $\delta_2 \geq \alpha_2$,
		we get $\de  \in A_{i+1} \cup \dots \cup A_N$; an analogous statement holds switching
		the components;
		\item[(5)] if $\al \ll \be \in \Ap(S)$ and they are consecutive in $S$, then there exists $i>0$ such that
		$\al \in  A_i$ and $\be \in A_{i+1}$; if $\al < \be \in \Ap(S)$, they share a component and 
		they are consecutive in $S$, then there exists $i>0$ such that either
		$\al \in  A_i$ and $\be \in A_{i+1}$ or $\al, \be \in A_i$;
		\item[(6)] let $\al \in A_i$ and let be $\be_1, \dots, \be_j$ all the elements of $\Ap(S)$, $\al<\be_r$ and
		consecutive to $\al$ in $\Ap(S)$. Then at least one of them belongs to $A_{i+1}$;
		\item[(7)] $\al=(\alpha_1,\gamma_2+e_2+1) \in A_i \Leftrightarrow (\alpha_1,n) \in A_i$ for some 
		$n\geq \gamma_2+e_2+1$ $\Leftrightarrow (\alpha_1,n) \in A_i$ for all $n\geq \gamma_2+e_2+1$; an analogous statement holds switching
		the components;
		\item[(8)] if $\al=(\alpha_1,\gamma_2+e_2+1) \in A_i$ and $\be =(\beta_1, \gamma_2+e_2+1) \in \Ap(S)$,
		with $\beta_1<\alpha_1$ and such that for every $a$, $\beta_1<a<\alpha_1$, the point 
		$(a,\gamma_2+e_2+1) \notin \Ap(S)$, then $\be \in A_{i-1}$; an analogous statement holds switching
		the components. (We could state this property saying that, definitively, consecutive vertical lines 
		(respectively, horizontal lines) of the Ap\'ery Set belong to consecutive levels.)
	\end{itemize}
\end{lemma}
\begin{proof} Properties (1) (2) and (3) follow immediately by definition.

\noindent (4) If $\delta_2 > \alpha_2$, then $\al \ll \de$ and the assertion follows by definition of levels.
If $\delta_2 = \alpha_2$, by property (3), $\de \in A_i \cup \dots \cup A_N$; if  $\de \in A_i$
we get a contradiction by the definition of levels, since $\al = \be \wedge \de$.

\noindent (5) Let $i$ be such that $\al \in A_i$. If $\al=(\alpha_1,\alpha_2) \ll \be=(\beta_1,\beta_2)$ 
and they are consecutive in $S$, 
then there are no elements $(a,b)$ of $S$ such that $\alpha_1 \leq a < \beta_1$ and $\alpha_2 < b$
or $\alpha_1 < a$ and $\alpha_2 \leq b < \beta_2$, since $(a,b)\wedge \be$ would be a point of $S$
 between
$\al$ and $\be$; in particular $\al$ cannot be obtained as infimum of points of $S$. Moreover $\be \in A_j$, with $j > i$. By property (3) all the points of $\Ap(S)$ 
bigger than or equal to $\be$ belong to $A_h$, with $h \geq j$; hence
$\al$ is maximal in $\Ap(S)\setminus (\bigcup_{h \geq j} A_h)$ with respect to $\le\le$ and, by definition of levels, this implies that
$j=i+1$.

Assume now that $\al$ and $\be$ share a component (e.g. $\al=(\alpha_1,\alpha_2), \be=(\alpha_1,\beta_2)$).
Since they are consecutive in $S$, there are no elements $(a,b)$ of $S$ such that $\alpha_1 \leq a$
and $\alpha_2 < b < \beta_2$. Let $i$ be such that $\al \in A_i$; by property (3) $\be \in  A_j$
with $j \geq i$ and all the points of $\Ap(S)$ 
bigger than or equal to $\be$ belong to $A_h$, with $h \geq j$; by definition of levels, $\be$
is maximal in $\Ap(S)\setminus (\bigcup_{h > j} A_h)$ with respect to $\le\le$ and therefore also $\al$ is maximal in the same set;
hence either $j=i$ or $j=i+1$.

\noindent (6) Assume by way of contradiction that for all $r$, $\be_r \in A_{h_r}$ with $h_r>i+1$, set $h=$min$\{h_r: r=1, \dots, j\}$. Hence
$\al$ is maximal in $\Ap(S)\setminus (\bigcup_{s \geq h} A_s)$ with respect to $\le\le$. In order to have $\al \in A_i$,
either $h=i+1$ (that is the thesis) or $h=i+2$ and $\al$ is obtained as infimum of two elements 
$\de_1, \de_2 \in A_{i+1}$.
But also in the second case, both $\de_i$ have to be consecutive to $\al$ (otherwise they would be bigger
than some elements $\be_r$, so they could not belong to $A_{i+1}$); hence we get a contradiction by the 
assumption that $h=i+2$. 
If, on the other hand, there exists $\be_r \in A_i$, it necessarily shares a component with $\al$. 
Hence, by property (4) all the other $\be_s$ belong to  $A_{i+1} \cup \dots \cup A_N$.
Now, applying the same argument as above, one of them has to be in $A_{i+1}$.

\noindent (7) Any element $\al =(\alpha_1, \alpha_2)$ with $\alpha_2 > \gamma_2$ 
(respectively, $\alpha_1 > \gamma_1$) belongs to $S$ if and only if 
$(\alpha_1, \gamma_2+1) \in S$ (respectively, $((\gamma_1+1, \alpha_2) \in S$). Hence it is clear that 
$\al=(\alpha_1,\gamma_2+e_2+1) \in \Ap(S)$ if and only if $(\alpha_1,n) \in \Ap(S)$ for some 
$n\geq \gamma_2+e_2+1$  if and only if $(\alpha_1,n) \in \Ap(S)$ for all $(\alpha_1,n)$ (and the analogous statement holds switching the components). 

So we have only to prove that these elements belongs to the same level (we will prove this fact for vertical lines and the corresponding statement for
horizontal line is analogous). If not, by property (5), there exist two elements $\al_1=(\alpha_1,\alpha_2) \in A_i$
and $ \al_2=(\alpha_1,\alpha_2+1)\in A_{i+1}$ consecutive in $\Ap(S)$. Now let $\beta_1> \alpha_1$ be the smallest integer such that 
$\be =(\beta_1, \alpha_2)\in \Ap(S)$; since $\al_1=\al_2\wedge \be$ and there are no elements $(a,\alpha_2)$ of $\Ap(S)$, with $\alpha_1 < a < \beta_1$, it is clear that either $\be \in A_i$ or 
$\be \in A_{i+1}$; moreover $\al_2 \ll (\beta_1, \alpha_2+2)$, hence the last belongs to $A_j$ with $j>i+1$.
Hence, also in the vertical line corresponding to $\beta_1$ there are elements on different levels.
Iterating the argument we get that the same happens for $\Delta_1(\g+\boldsymbol{e}) \subseteq A_N$;
a contradiction.

\noindent (8) By property (6), alle the elements of $S$ above $\al$ are in $A_i$. Hence 
$\be \ll (\alpha_1, \gamma_2+e_2+ 2)$ and therefore it belongs to $A_j$ with $j <i$. Moreover
the hypothesis implies that $\be$ is maximal in $\Ap(S)\setminus (\bigcup_{h \geq i} A_h)$ with respect to $\le\le$
and cannot be obtained as infimum of two other elements maximal in the same set. 
The thesis follows immediately. 
\end{proof}

Next lemma describes global properties of the elements of a good semigroup $S$ and of its Ap\'{e}ry Set.

\begin{lemma}\label{altro-Ap} The following assertions hold:
	\begin{itemize}
		\item[(1)]  Let $\al \in \N^2$ and assume there is a finite positive number of elements in $ \Delta^S_1(\al) \cap (\e+S)$. Call $\de$ the maximum of them. Hence $ \Delta^S(\de) \subseteq \Ap(S) $;
		\item[(2)] Let $ \al \in \Ap(S) $. If there exists $ \be \in (\e+S) \cap \Delta_1(\al) $, then $ \Delta_2^S(\al) \subseteq \Ap(S) $;
		\item[(3)] Let $ \al=(a_1, a_2) \in A_i $ and suppose there exists $b_2 < a_2$ such that $\de=(a_1, b_2) \in S$ and $ \Delta_2^S(\de)  \subseteq \Ap(S) $. Then the minimal element $\be=(b_1, b_2)$ of $ \Delta_2^S(\de)   $ is in $A_j$ for some $j \leq i$. In particular, if $ \Delta^S(\de)  \subseteq \Ap(S) $ and  $\al$ is the minimal element of $ \Delta_1^S(\de) $, $\be \in A_i.$
		\item[(4)] Let $ \al=(a_1, a_2) \in A_i $ and suppose there exists $\de \in (\e+S) \cap \Delta_1(\al)$. Then $ \Delta_2^S(\al)  \subseteq \Ap(S) $ and the minimal element $\be=(b_1, a_2)$ of $ \Delta_2^S(\al) $ is also in $A_i$.
		\item[(5)] Let $\al \in A_i$ and assume $ \Delta^S_1(\al) \subseteq \Ap(S) $. Assume also that there exists $ \be \in \Delta^S_1(\al) \cap A_{i+1} $. Then there exists $\te \in (\Delta^S_1(\al) \cap A_i) \cup \lbrace  \al \rbrace$ such that $\te < \be$ and $ \Delta^S(\te) \subseteq \Ap(S) $. 
		
		The analogous assertions hold switching the components.
			\end{itemize}
\end{lemma}

\begin{proof} 
(1) Since $\de$ is the maximum of $ \Delta^S_1(\al) \cap (\e+S)$, we have $ \Delta^S_1(\de) \subseteq \Ap(S) $. Now, since $(\e+S)$ is a good ideal of $S$, by property (G2), also $ \Delta^S_2(\de) \subseteq \Ap(S) $.

\noindent (2) Assume that there exists $ \de \in \Delta_2(\al) \cap (\e+S) $. Then again, since $(\e+S)$ is a good ideal, by property (G1), $\al= \be \wedge \de \in (\e+S)$ and this is a contradiction.

\noindent (3) First we notice that $ \Delta_2^S(\de)  $ is non-empty since also $ \Delta_1^S(\de)  $ is non-empty (by axiom (G2)). Now, if $i=N$, the thesis easily follows by Remark \ref{rem1}.
For $i < N$, we use the following argument: by definition of $\Ap(S)$ we can always find an element $ \te=(g_{1},g_{2}) \in A_{i+1} $ with $ g_{1} > a_1 $ and $ g_{2} \geq a_2 $. Hence, the fact that $ \be $ is the minimal element of $ \Delta_2^S(\de) $ implies that $ g_{1} \geq b_1 $ and this implies $ \be \in A_{j} $ for $ j \leq i+1 $. 

If we assume by way of contradiction $ \be \in A_{i+1} $ we would have $ g_{1} = b_1 $ and hence, by axiom (G2), there exists an element $ \om = (h_1, b_2)$ with $h_1 > b_1$. Since $ \Delta_2^S(\de)  \subseteq \Ap(S) $, we have $ \om \in \Ap(S) $ and we may assume $ \om $ minimal in $ \Delta_2^S(\be) $. Thus, if $ \om \in A_{i+1} $ we have $\be= \te \wedge \om \in A_i$, otherwise we should have $ \om \in A_{j} $ for some $j > i+1$. But now we are in the same situation of the hypothesis of the lemma with $ \te, \be, \om $ playing the role of $ \al, \de, \be $. In this way, iterating the process, we can create an infinite sequence of elements $ \om^k \in \Delta_2(\de) \cap A_{i+k} $ and this is a contradiction because the levels of $\Ap(S)$ are in a finite number. The last sentence of the statement follows since, having also $ \Delta_1^S(\de)  \subseteq \Ap(S) $, we apply the same result and get the level of $\al$ less or equal than the level of $\be.$

\noindent (4) If $i=N$, the thesis is clear by Remark \ref{rem1}. For $i < N,$ clearly $ \Delta_2^S(\al) $ is non-empty and it is contained in $\Ap(S)$ by (2). First assume that there exists $\te=(g_1, g_2) \in A_{i+1}$ such that $ \te \gg \al $. Since $\be $ is the minimal element of $ \Delta_2^S(\al)  $, we have $g_1 \geq b_1.$ But now, if $g_1 > b_1$, then $ \te \gg \be $ and hence $ \be \in A_i $. Instead if $g_1=b_1$ we can find a minimal element $\om \in \Delta_2(\be) \cap \Ap(S)$ and $\be =  \te \wedge \om$. By (3), $\om \in A_j$ with $j \leq i+1$ and thus $\be \in A_i.$

Now assume that there is no element of $A_{i+1}$ dominating $\al$. 
It follows that $\al =  \te \wedge \be$ with $\te, \be \in A_{i+1}$. We can assume by way of contradiction both of these elements to be minimal, otherwise we would have the minimal $\be \in A_i $. If $\te \in \Delta_1(\al)$ is still such that $\te < \de$, we have $ \Delta_2^S(\te)  \subseteq \Ap(S) $ and we can find in it a minimal element $\om \in \Delta_2^S(\te) $ which has to be in $A_{i+2}$ by Lemma \ref{basic-Ap}(5) (otherwise we would have an element of $A_{i+1}$ dominating $\al$).

Suppose there exists $\om^1 \in A_{i+2}$ such that $\om^1 \gg \te$. Clearly we cannot have neither $ \om^1 \gg \om $ nor $ \om^1 \in \Delta_1(\om) $ because this would contradict (3). Hence there exists $\om^1 \wedge \om \in \Delta_2^S(\te) $ and this is a contradiction since $\om $ is minimal in $ \Delta_2^S(\te) $.     

Hence, we can assume that there are not elements of $ A_{i+2} $ dominating $ \te $. This means that $ \te $ is the minimum of  two elements $ \om^1, \he^2 \in A_{i+2}$ and we can iterate the process replacing the elements $\te, \al , \be$ with $ \om^1, \te, \he^2 $. If $ \Delta_1(\al) $ is eventually contained in $\e+S$, after a finite number of iteration we will find an element $\om^k$ which is maximal in $ \Delta_1(\al)  \cap \Ap(S)$ and it is dominated by some element of a greater level of $\Ap(S)$ (notice that in this case the elements in $ \Delta_1(\al) $ are dominated by elements in $A_N$). This would lead to a contradiction like in the previous paragraph.
If $ \Delta_1(\al) \cap (\e+S) $ has a maximum $\de$, then $ \Delta^S(\de) \subseteq \Ap(S) $ by (1) and our iterative process will end, by replacing the name of the elements and of the levels, in a situation with $\al =  \te \wedge \be$ with $\te, \be \in A_{i+1}$ and $\te > \de > \al$. In this setting, we can find a minimal element $\om \in \Delta_2(\de) \cap \Ap(S)$. By (3), $\om \in A_j$ with $j \leq i+1$ and this again contradicts the assumption since $ \om \gg \al. $

\noindent (5) Assume $ \Delta^S_2(\al) \nsubseteq \Ap(S) $. Hence by (4), we have that the minimal element $\om$ of $ \Delta_1^S(\al) $ is also in $A_i$. Hence $\om < \be$. If again $ \Delta^S_2(\om) \nsubseteq \Ap(S) $ we find an element $ \om^1 \in \Delta_1^S(\al) \cap A_i $ and $ \om^1 < \be $. Since there is a finite number of elements in $ \Delta_1^S(\al) $ between $\al$ and $\be$, the process must stop to an element $\te \in \Delta_1^S(\al) \cap A_i$ such that $\te < \be$ and $ \Delta^S(\te) \subseteq \Ap(S) $.
\end{proof}

\medskip

\begin{theorem} \label{livelli}
Let $S \subseteq \N^2$ be a good semigroup and let $\boldsymbol{e}=(e_1, e_2)$ be its minimal non zero element. Let $\Ap(S)=\bigcup_{i=1}^N A_{i}$ where the sets $A_{i}$ are defined as above.
Then $N=e_1+e_2$.
\end{theorem}

\begin{proof}
We have that $\Ap(S)=S\setminus (\boldsymbol{e}+S)$. Moreover both $S$ and $\boldsymbol{e}+S$
are good ideals so we can compute the distance function $d(S\setminus \boldsymbol{e}+S)$ as
$d_S(\boldsymbol{0},\g+\boldsymbol{e}+\boldsymbol{1})-d_{\boldsymbol{e}+S}(\boldsymbol{e}, \g+\boldsymbol{e}+\boldsymbol{1})$;
on the other hand we know that $d(S\setminus \boldsymbol{e}+S)=e_1+e_2$.

Hence, to prove that $N\geq e_1+e_2$ we show that there exists a saturated chain in $S$, 
between $\boldsymbol{0}$ and $\g+\boldsymbol{e}+\boldsymbol{1}$ that contains exactly one element of every level $A_i$: if we do not consider the $N$ elements of $\Ap(S)$ in this chain we get a chain 
(not necessarily saturated) in $\boldsymbol{e}+S$; hence $d_S(\boldsymbol{0},\g+\boldsymbol{e}+\boldsymbol{1})-N
 \leq d_{\boldsymbol{e}+S}(\boldsymbol{e}, \g+\boldsymbol{e}+\boldsymbol{1})$, that means $$e_1+e_2=d_S(\boldsymbol{0},\g+\boldsymbol{e}+\boldsymbol{1})-d_{\boldsymbol{e}+S}(\boldsymbol{e}, \g+\boldsymbol{e}+\boldsymbol{1})\leq N.$$
To construct such a chain we start from $\boldsymbol{0}\in S \cap \Ap(S)$ and then we choose
$N$ elements, one for each level $A_i$ using property (6) of Lemma \ref{basic-Ap}: given $\al_i \in A_i$
we choose $\al_{i+1} \in A_{i+1}$ consecutive to $\al_i$ in $\Ap(S)$;
so we get a chain of $N$ elements in $\Ap(S)$, each one consecutive to the previous in $\Ap(S)$.
Hence when we saturate this chain in $S$, we add only elements in $S\setminus \Ap(S)=\boldsymbol{e}+S$
and we obtain the desired chain.

\medskip 

In order to prove that $N \leq e_1+e_2$ we want to construct a saturated chain in $\boldsymbol{e}+S$
between $\boldsymbol{e}$ and $\g+\boldsymbol{e}+\boldsymbol{1}$
such that, when we saturate it in $S$, as a chain between $\boldsymbol{0}$ and $\g+\boldsymbol{e}+\boldsymbol{1}$, we use at least one element for every level $A_i$
(it is clear that we can only add elements of $\Ap(S)=S\setminus (\boldsymbol{e}+S)$:
in fact, if we add $n\geq N$ elements in $\Ap(S)$, this would imply $d_S(\boldsymbol{0},\g+\boldsymbol{e}+\boldsymbol{1})= d_{\boldsymbol{e}+S}(\boldsymbol{e}, \g+\boldsymbol{e}+\boldsymbol{1})+n \geq d_{\boldsymbol{e}+S}(\boldsymbol{e}, \g+\boldsymbol{e}+\boldsymbol{1})+N$, that is 
$N \leq d_S(\boldsymbol{0},\g+\boldsymbol{e}+\boldsymbol{1})- d_{\boldsymbol{e}+S}(\boldsymbol{e}, \g+\boldsymbol{e}+\boldsymbol{1})=e_1+e_2$.

To construct such a chain we start with $\boldsymbol{0} \ll \boldsymbol{e}$, that is
a saturated chain in $S$; hence we can assume that we have
constructed a saturated chain in $S$, say $\al_0=\boldsymbol{0}<\al_1< \dots <\al_h$,
such that, $\al_h \leq \g+\boldsymbol{e}+\boldsymbol{1}$, if we delete the elements $\al_i \in \Ap(S)$ we get a saturated chain in $\boldsymbol{e}+S$ and every level $A_1, \dots, A_j$ has at least one representative in it.
To apply a recursive argument we need either to stretch the chain adding one or more new elements,
the first consecutive to 
$\al_h$ in $S$ and any of the others consecutive to the previous one,
or to produce a new chain with the same properties (replacing the last elements of the constructed chain) for which all the levels $A_1, \dots, A_{j+1}$ have at least one representative in it.
The process will end, since the length of a saturated chain is bounded 
by $d_S(\boldsymbol{0},\g+\boldsymbol{e}+\boldsymbol{1})$ (and at the last step we have to
touch $A_N$) and the number of levels is bounded by $N$. We explain now how to add a new element to the chain in all the different possible cases. Before starting, we observe that if we have $ \g + \e + 1 \in \Delta^S(\al_h) $ we complete the chain adding all the elements between $ \al_h $ and $ \g + \e + 1 $ sharing a coordinate with them (they are all on the same horizontal or on the same vertical line). By Lemma \ref{basic-Ap}(7 and 8) this chain will touch exactly once all the levels of $\Ap(S)$ between $j+1$ and $N$. 

Hence we consider now all the cases in which $\al_h \ll \g+\e+1.$
We can start from one of the
following two cases: either (case A) $\al_{h} \in \boldsymbol{e}+S$ or 
(case B) $\al_h \in A_j$ 
(notice that at the beginning of the chain we have $\boldsymbol{0} \ll \boldsymbol{e}$, 
so we start from case A). 

In both cases if $\al_h$ has only one consecutive element $\be$ in $S$, necessarily (by axiom (G2))
we have $\al_{h} \ll \be$ and
we should be forced to choose $\al_{h+1}=\be$.  If $\be \in \boldsymbol{e}+S$,
the new chain obviously satisfy the requested properties, has one element more and now we are in case A.
On the other hand, assume $\be \in \Ap(S)$; the condition
$\al_{h} \ll \be$ implies that $\al_{i} \ll \be$, for all $i=0, \dots, h$; in particular, let 
$\al_r \in  A_j$ be the last element of the Ap\'ery Set in the chain. Hence, by Remark \ref{rem1}
we know that $\be \in A_{j+1} \cup \dots \cup A_N$; if it is in $A_{j+1}$, we simply add $\be$ to the chain and proceed in case B.
Otherwise, if $\be$ it is not in $A_{j+1}$, by Lemma \ref{basic-Ap}(6) there exists another 
element $\de$ consecutive in $\Ap(S)$ to $\al_r$, such that $\de \in A_{j+1}$
(notice that, by Lemma \ref{basic-Ap}(5), this situation can happen only in case A).
Now $\de$ has to share a component with one of the elements $\al_r, \dots, \al_h$
otherwise, taking infimums, it would create a new point that makes the original chain non-saturated; 
more precisely, if it is above (respectively, on the right) of some $\al_l$ for $r \leq l \leq h$, then $\al_{l+1}$ has to be either above or on the right of $\al_l$.
Now we change the chain substituting $\al_m$ with $\de \wedge \al_m$, for every $m \geq l+1$;
successively we add to the chain $\de \wedge \be$ and all the other elements of $S$ 
on the vertical (respectively, horizontal) line until we reach $\de$ (notice that we may have $\de \wedge \al_m = \de$ for some $m$ and in that case we simply stop to $\de$ because we have reached it). We show an example of the preceding process in the first picture of Figure \ref{mini1}.


Hence we created a new chain with the requested 
properties, such that every level $A_1, \dots, A_{j+1}$
has at least one representative in it and now we are in case B.

It remains to study what happens if $\al_h = \be_1 \wedge \be_2$ 
(where both $\be_i$ are consecutive to $\al_h$ in $S$).

\noindent $\bullet$ Case A. \\
In this case let $\al_r \in  A_j$ be the last element of the Ap\'ery Set in the chain. The following situations can occur: \\
A.1) both $\be_i \in \Ap(S)$: if at least one $\be_i$ belongs to $A_{j+1}$, we set $\al_{h+1}:= \be_i$ and we switch to
case B. Otherwise, since $\al_r \ll \be_i$, for at least one $i=1,2$, we have, for the same $i$, $\be_i \in A_{j+2}\cup \dots \cup A_N$. By Lemma \ref{basic-Ap}(6), there exists another element $\de$ consecutive in $\Ap(S)$ to $\al_r$, such that $\de \in A_{j+1}$. Using the same argument as above, i.e. replacing the last part of the chain with $\de \wedge \al_m$ (for $ r < m \leq h$) and then proceeding on a single line until reaching $\de$, we obtain the desired result and we switch to case B; \\
A.2) both $\be_i \in \boldsymbol{e}+S$: we can move to any one of them indifferently and proceed in case A;\\
A.3) $\be_1 \in \Ap(S)$ and $\be_2 \in \boldsymbol{e}+S$: if $\be_1 \in A_j$ (this can happen only if the last element of the Ap\'ery Set in the chain $\al_r$ shares a component with both $\al_h$ and $\be_1$), we set $\al_{h+1}=\be_2$ and we proceed in case A; if $\be_1 \in A_{j+2}$, we take again $\de$ consecutive in $\Ap(S)$ to $\al_r$ such that $\de \in A_{j+1}$ and, replacing the last part of the chain with $\de \wedge \al_m$ (the elements $ \al_m$ are defined as in A.1), we obtain the desired result and we switch to case B.

It remains the hardest situation, i.e. $\be_1 \in A_{j+1}$; in this case we have to show that both $\be_1$ and $\be_2$ do not have the same element as unique consecutive in $S$; 
if this was the case and if $\be_1$ was the only 
consecutive of $\al_r$  in $\Ap(S)$ belonging to $A_{j+1}$, it would be impossible to proceed,
because either we would skip the level $A_{j+1}$ or we would create a chain such that, if we delete the elements of $\Ap(S)$, we do not get a saturated chain in $\boldsymbol{e}+S$ (for this situation, see the second picture of figure \ref{mini1}, in which we denote by $\te$ the unique consecutive element in $S$ of both $\be_1$ and $\be_2$ and we consider the two possible chains between $\al_h$ and $\te$).

\begin{figure}[h]
\setlength{\unitlength}{0.4mm}
\begin{center}
\begin{picture}(90,100)(90,10)
\put(-10,0){\line(1,0){100}}
\put(-10,0){\line(0,1){100}}
\put(10,20){$\circ$}
\put(10,16){ \tiny $\al_r$}
\put(20,40){$\bullet$}
\put(16,36){ \tiny $\al_{r+1}$}
\put(20,60){$\bullet$}
\put(3,55){ \tiny $\al_{h-1} \wedge \de$}
\put(20,80){$\bullet$}
\put(3,75){ \tiny $\al_{h} \wedge \de$}
\put(20,90){$\circ$}
\put(12,90){ \tiny $\de$}
\put(50,80){$\bullet$}
\put(45,75){ \tiny $\al_{h}$}
\put(40,40){$\bullet$}
\put(38,36){ \tiny $\al_{r+2}$}
\put(50,60){$\bullet$}
\put(45,55){ \tiny $\al_{h-1}$}
\put(70,60){$\bullet$}
\put(70,90){$\circ$}
\put(65,86){ \tiny $\be$}
\thicklines
\thinlines
\put(180,0){\line(1,0){100}}
\put(180,0){\line(0,1){100}}
\put(230,20){$\circ$}
\put(230,16){ \tiny $\be_1$}
\put(200,20){$\bullet$}
\put(200,16){ \tiny $\al_{h}$}
\put(200,40){$\bullet$}
\put(200,36){ \tiny $\be_2$}
\put(250,50){$\bullet$}
\put(250,46){ \tiny $\te$}
\put(260,60){$\bullet$}
\end{picture}
\end{center}
\caption{\scriptsize Explanation of parts of the proof.  
 } 
\label{mini1}
\end{figure}
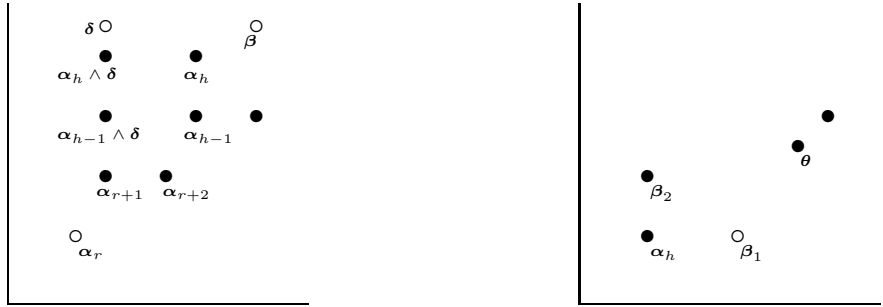


Let assume that $\al_h$ and $\be_2$ share the first component (the other case is symmetric):
So $\al_h=(a_1,a_2)$ and $\be_2=(a_1, b_2)$, with $b_2 > a_2$. They do not belong to $\Ap(S)$
hence both $\al_h-\boldsymbol{e}$ and $\be_2-\boldsymbol{e}$ belong to $S$. Hence, by Property (G2),
there must be an element $(c_1, a_2-e_2) \in S$ and so we have also another point of $S$,
$\de=(c_1+e_1, a_2) \in \e+S$, on the right of $\al_h$. Since $\be_1$ is also on the right of $\al_h$ 
and it is consecutive to it in $S$, we have three points in the same horizontal line:
$\al_h < \be_1 < \de$. Now, choosing $\de$ minimal, we are sure that it is consecutive to $\al_h$
in $\boldsymbol{e}+S$. Moreover, we are sure that
moving from $\be_1$ to $\de$ on the horizontal line, if we meet points of $\Ap(S)$, by Lemma \ref{basic-Ap}(5) we do not skip 
any level and by Lemma \ref{altro-Ap}(4) and Lemma \ref{basic-Ap}(4) we do not repeat twice the same level. Hence we can stretch the chain up to $\de$ and proceed in case A. 


\noindent $\bullet$ Case B. We notice that, in this case, if $\be_i\in \Ap(S)$ it has to belong either to $A_j$ or to $A_{j+1}$, since they are consecutive  in $S$ to $\al_h \in \Ap(S)$ (again by Lemma \ref{basic-Ap}(5)).\\
The following situations can occur: \\
B.1) both $\be_i \in A_{j+1}$: we can move to any one of them indifferently and proceed in case B;\\
B.2) both $\be_i \in \boldsymbol{e}+S$: this cannot happen by Lemma \ref{altro-Ap}(2); \\
B.3) $\be_1 \in A_j$ and $\be_2 \in A_{j+1}$: we move to $\be_2$ and proceed in case B;\\
B.4) $\be_1 \in A_j$ and $\be_2 \in \boldsymbol{e}+S$: we move to $\be_2$ and switch to case A;\\
B.5) $\be_1 \in A_{j+1}$ and $\be_2 \in \boldsymbol{e}+S$: we can assume that $\be_1$ is 
on the right of $\al_h$ and $\be_2$ is above it; in this case 
either there is another 
element $\be' \in A_{j+1}$ such that $\al = \be_1 \wedge \be'$ and $\al <\be_2 < \be' $ share the first
component; we choose $\be'$ minimal with this property and we move to $\be_2$ and then to $\be'$ (considering all possible elements between them, that have to belong to $\boldsymbol{e}+S$)
or there exists $\be'' \in A_{j+1}$ above $\be_1$ and consecutive to both $\be_i$; so we move to 
$\be_2$ and then to $\be''$. In both cases we proceed in case B.

The proof is complete.
\end{proof}

From now on we are going to denote the number of levels of the Ap\'{e}ry Set by $e= e(S)= e_1 + e_2$, that as we noticed in the previous section,
coincides with the multiplicity of the ring $R$, in case $S=v(R)$.

We derive from the proof of the preceding theorem, a sort of converse of property (1) of Lemma \ref{basic-Ap}. We are going to make use of this next result in the last section of this article, while proving a duality property for the levels of the Ap\'{e}ry Set of a symmetric good semigroup.

\begin{prop} 
\label{caso3} 
Let $S \subseteq \N^2$ be a good semigroup and let $\Ap(S)=\bigcup_{i=1}^e A_{i}$ be its Ap\'{e}ry Set. Let $\al \in A_i$ for $i \geq 2$. Then, there exists $\be \in A_{i-1}$ such that $\be \leq \al.$
\end{prop}

\begin{proof}
Since $A_1= \lbrace \boldsymbol{0} \rbrace$, and $\al \geq \boldsymbol{0}$ for every $\al \in S$, the thesis is true for $i = 2$ and hence, by induction, we can assume it true for every $j < i.$ Assume by way of contradiction that there exists $\al \in A_i$ such that $\te \not \leq \al$ for every $\te \in A_{i-1}$. We can further assume that also $\de \not \leq \al$ for every $\de \in A_{i}$, otherwise we can simply replace $\al$ with some element $\de \leq \al$ and minimal in $A_i$ with respect to ''$\leq $''. 

Take $\om \in \Ap(S)$ such that $\om \leq \al$ and they are consecutive in $ \Ap(S) $, hence $\om \in A_j$ with $j < i-1.$ We may assume $j$ to be the maximal level of an element of $ \Ap(S) $ having $\al$ as a consecutive element in $ \Ap(S) $. Assume there exists an element $ \om^{\prime} \in A_j $ such that $\om \in \Delta^S(\om^{\prime})$. Hence, we can find a saturated chain in $S$ between $ \om^{\prime} $ and $\al$ that does not contain any other elements of $\Ap(S)$. Indeed, we can find $\be \in S$ such that $\om^{\prime} \leq \be \leq \al$, $\be \in \Delta^S(\om^{\prime})$ and it is incomparable with $\om$ (i.e. either $\be \in \Delta^S_2(\om^{\prime}) $ and $\om \in \Delta^S_1(\om^{\prime})$ or the converse). If $\be \in \Ap(S)$, then it has to be in $A_j$, but this is impossible since $\om^{\prime}= \om \wedge \be$ would be the minimum of two elements of $A_j$, and therefore not in $A_j$.

It follows that we can choose an element $ \tilde{\om} \leq \om $ minimal with respect to the property of being in $A_j$ (this element could be $\om$ itself) and find a saturated chain in $S$ between $ \tilde{\om} $ and $\al$ not containing any other elements of $\Ap(S)$. By inductive hypothesis, $\tilde{\om} \geq \de \in A_{j-1}$, and hence we can iterate the preceding process and construct a saturated chain in $S$ between $\boldsymbol{0}$ and $\al$, containing only one element for every level of $\Ap(S)$ between $1$ and $j$ and not containing any element in the levels strictly between $j$ and $i$. As in the first part of the proof of Theorem \ref{livelli}, we can extend this chain adding a chain in $S$ from $\al$ to $\g - \e + \boldsymbol{1}$ including only one element for each level of $\Ap(S)$ greater than $i$. The obtained chain going between $\boldsymbol{0}$ and $\g +\e +\boldsymbol{1}$ contains $h := e - (i-j) + 1 $ elements of $\Ap(S)$, thus, removing those elements, we can find a chain in $\e+S$ between $\boldsymbol{e}$ and $\g +\e +\boldsymbol{1}$ of length $$ d_S(\boldsymbol{0}, \g +\e +\boldsymbol{1}) - h = d_{\e+S}(\boldsymbol{e}, \g +\e +\boldsymbol{1}) + e  - h. $$ Since $j < i-1$, this length is strictly bigger than $ d_{\e+S}(\boldsymbol{e}, \g +\e +\boldsymbol{1}) $ and this is a contradiction.
\end{proof}

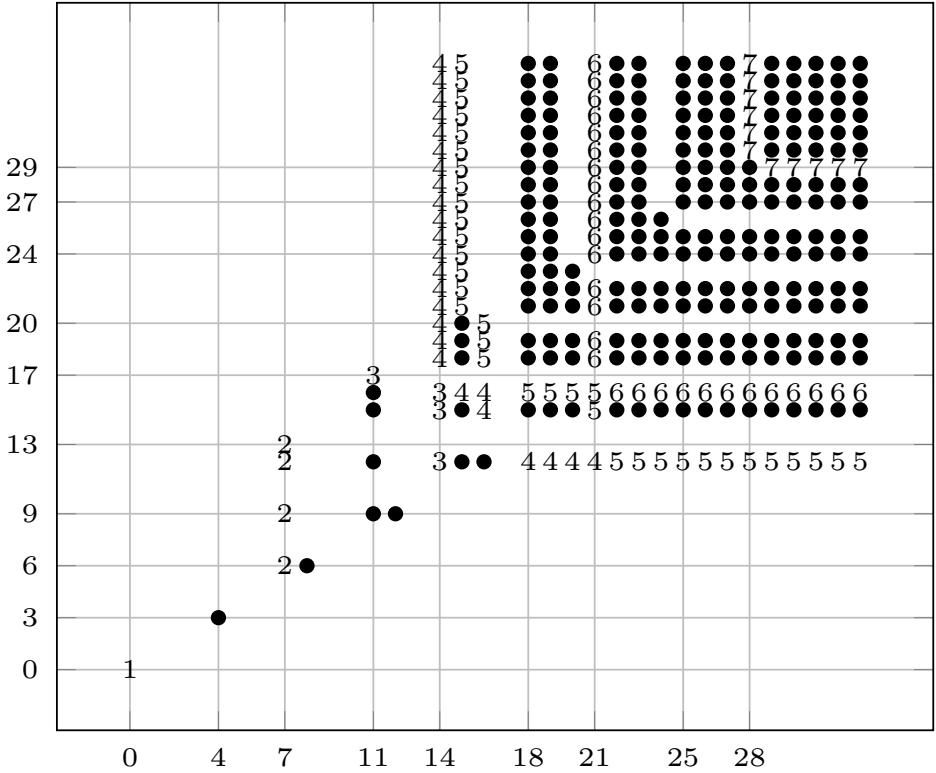
\begin{figure}[h] 
 
 \begin{tikzpicture}[scale=1.7]
	\begin{axis}[grid=major, ytick={0,3,6,9,13, 17, 20, 24, 27, 29}, xtick={0,4,7,11,14,18, 21, 25, 28}, yticklabel style={font=\tiny}, xticklabel style={font=\tiny}]
	

\addplot[only marks] coordinates{  (4,3)  (8,6) (11,9) (11,12) (11,15) (11,16) (12,9)  (15,12) (15,15)  (15,18)
 (15,19) (15,20) (16,12)    (18,15)  (18,18) (18,19) (18,21)
 (18,22) (18,23) (18,24) (18,25) (18,26) (18,27) (18,28) (18,29) (19,15)  (19,18) (19,19) (19,21) (19,22) (19,23) (19,24) (19,25) (19,26) (19,27) (19,28) (19,29) (20,15) 
 (20,18) (20,19) (20,21) (20,22) (20,23)  (22,15)  (22,18) (22,19)
 (22,21) (22,22) (22,24) (22,25) (22,26) (22,27) (22,28) (22,29) (23,15)  (23,18) (23,19) (23,21) (23,22) (23,24) (23,25) (23,26) (23,27) (23,28) (23,29) (24,15)  (24,18)
 (24,19) (24,21) (24,22) (24,24) (24,25) (24,26)  (25,15)   (25,18) (25,19) (25,21) (25,22) (25,24) (25,25) (25,27) (25,28) (25,29) (26,28) (26,29) (27,28) (27,29)  (28,28)(28,29) (26,27) (27,27) (28,27) (26,25) (27,25) (28,25) (26,24) (27,24) (28,24) (26,22) (27,22) (28,22) (26,21) (27,21) (28,21) (26,19) (27,19) (28,19) (26,18) (27,18) (28,18) (26,15) (27,15) (28,15) (29,15) (30,15) (31,15) (32,15) (33,15) (29,18) (30,18) (31,18) (32,18) (33,18) (29,21) (30,21) (31,21) (32,21) (33,21) (29,19) (30,19) (31,19) (32,19) (33,19)
 (29,22) (30,22) (31,22) (32,22) (33,22)
 (29,24) (30,24) (31,24) (32,24) (33,24)
 (29,25) (30,25) (31,25) (32,25) (33,25)
 (29,27) (30,27) (31,27) (32,27) (33,27)
 (29,28) (30,28) (31,28) (32,28) (33,28)
 (29,30) (30,30) (31,30) (32,30) (33,30)
 (29,31) (30,31) (31,31) (32,31) (33,31)
 (29,32) (30,32) (31,32) (32,32) (33,32)
 (29,33) (30,33) (31,33) (32,33) (33,33)
 (29,34) (30,34) (31,34) (32,34) (33,34)
 (29,35) (30,35) (31,35) (32,35) (33,35)
 (18,30)(18,31) (18,32) (18,33) (18,34) (18,35)
 (19,30)(19,31) (19,32) (19,33) (19,34) (19,35)
 (22,30)(22,31) (22,32) (22,33) (22,34) (22,35)
 (23,30)(23,31) (23,32) (23,33) (23,34) (23,35)
 (25,30)(25,31) (25,32) (25,33) (25,34) (25,35)
 (26,30)(26,31) (26,32) (26,33) (26,34) (26,35)
 (27,30)(27,31) (27,32) (27,33) (27,34) (27,35)}; 
 \addplot[only marks, mark=text, mark options={scale=1,text mark={ \tiny 7}}, text mark as node=true] coordinates{(28,30) (28,31) (28,32) (28,33) (28,34) (28,35) (29,29) (30,29) (31,29) (32,29) (33,29)}; 
 \addplot[only marks, mark=text, mark options={scale=1,text mark={ \tiny 1}}, text mark as node=true] coordinates{(0,0)}; 
 \addplot[only marks, mark=text, mark options={scale=1,text mark={ \tiny 2}}, text mark as node=true] coordinates{(7,6) (7,9) (7,12) (7,13)}; 
 \addplot[only marks, mark=text, mark options={scale=1,text mark={ \tiny 3}}, text mark as node=true] coordinates{(11,17) (14,12) (14,15) (14,16)}; 
 \addplot[only marks, mark=text, mark options={scale=1,text mark={ \tiny 6}}, text mark as node=true] coordinates{  (21,18) (21,19) (21,21) (21,22) (21,24) (21,25) (21,26) (21,27)(21,28) (21,29) (21,30) (21,31) (21,32) (21,33) (21,34) (21,35) (22,16) (23,16) (24,16) (25,16) (26,16) (27,16) (28,16) (29,16) (30,16) (31,16) (32,16) (33,16) }; 
 \addplot[only marks, mark=text, mark options={scale=1,text mark={ \tiny 4}}, text mark as node=true] coordinates{(18,12)(19,12)(20,12)(21,12)(16,15) (16,16)(15,16)(14,18) (14,19) (14,20) (14,21) (14,22) (14,23) (14,24) (14,25) (14,26) (14,27) (14,28) (14,29) (14,30) (14,31) (14,32) (14,33) (14,34) (14,35) }; 
 \addplot[only marks, mark=text, mark options={scale=1,text mark={ \tiny 5}}, text mark as node=true] coordinates{(21,15) (21,16) (22,12) (23,12) (24,12) (25,12) (26,12) (27,12) (28,12) (29,12) (30,12) (31,12) (32,12) (33,12) (18,16) (19,16) (20,16) (16,18) (16,19) (16,20) (15,21) (15,22) (15,23) (15,24) (15,25) (15,26) (15,27) (15,28) (15,29) (15,30) (15,31) (15,32) (15,33) (15,34) (15,35)}; 
 \end{axis}
  \end{tikzpicture}
  \caption{\scriptsize This is an example of a good semigroup that is not the value semigroup of any ring, see \cite[Example 2.16, pag.8]{a-u}. 
  }
\label{semdanna}
\end{figure}

As we can observe in all the preceding examples of good semigroups, the first levels of $\Ap(S)$ are finite while the others contain either one or two infinite lines of elements. After formalizing the concept of infinite lines of elements in two definitions, we describe precisely this behavior in the next proposition.

\begin{definition} 
\label{d3} \rm
 Let $S \subseteq S_1 \times S_2 \subseteq \N^2$ be a good semigroup. 
 Given an element $s_1 \in S_1$, we say that $ \Delta_1(s_1,r) $ is an \it infinite line \rm of $S$ if there exists $r \in S_2$ such that $ \Delta_1(s_1, r) \subseteq S$. If $ \Delta_1(s_1, r) \subseteq \Ap(S)$, we say that $ \Delta_1(s_1,r) $ is an infinite line of $\Ap(S)$. 
 
 Analogously, given an element $s_2 \in S_2$, $\Delta_2(q,s_2)$ is an \it infinite line \rm of $S$ (resp. $\Ap(S)$) if there exists $q \in S_1$ such that $ \Delta_2(q, s_2) \subseteq S$ (resp. $\Ap(S)$).
 If an infinite line of $S$ is not an infinite line of $\Ap(S)$, then it is an infinite line of $\e+S.$
 \end{definition}

\begin{definition} 
\label{d4}
 \rm Let $S \subseteq S_1 \times S_2 \subseteq \N^2$ be a good semigroup and let $ \Ap(S)=\bigcup_{i=1}^e A_i $ be its Ap\'{e}ry Set. For $i=1, \ldots, e$, we say that: 
 \begin{enumerate}
 \item $A_i$ contains two infinite lines if there exist two elements $s_1 \in S_1$ and $s_2 \in S_2$, such that, for some $q \in S_1, r \in S_2$, $\Delta_1(s_1,r) $, $\Delta_2(q, s_2)$ are infinite lines of $\Ap(S)$ and they are contained in $A_i$.
 \item $A_i$ contains only one infinite line if only one of the previous conditions hold. 
 \item $A_i$ is finite if it contains a finite number of elements.
 \end{enumerate}
 \end{definition}

\begin{theorem}
\label{infiniti}
Let $S \subseteq S_1 \times S_2 \subseteq \N^2$ be a good semigroup, let $\e=(e_1, e_2)$ be its minimal non-zero element. Let $ \Ap(S)=\bigcup_{i=1}^e A_i $ be the Ap\'{e}ry Set of $S$. Assume $e_1 \geq e_2$. Then:
\begin{itemize}
 \item[(1)] The levels $A_e, A_{e-1}, \dots, A_{e-e_2+1}$ contain two infinite lines.
 \item[(2)] The levels $A_{e-e_2}, \dots, A_{e-e_1+1}$ contain only one infinite line of the form $ \Delta_1(s_1,r)  $ corresponding to some element $s_1 \in S_1$.
 \item[(3)] The levels $A_{e-e_1}, \dots, A_2, A_{1}$ are finite.
 \end{itemize}
 If $e_1 \leq e_2$ the correspondent analogous conditions hold.
\end{theorem} 


\begin{proof}
First we show that in the projection $S_1$ of $S$ there are 
exactly $e_1$ elements $s_1, \ldots, s_{e_1}$ such that $\Delta_1(s_i, r)$ is an infinite line of $\Ap(S)$ (for some $r \in S_2$). 
Let $\boldsymbol{c}=(c_1,c_2)$ be the conductor of $S$. Following the preceding definitions, we have that for every $n \geq c_1$ and sufficiently large $r \in S_2$, $ \Delta_1(n,r) $ is an infinite line of $S$. Moreover, $ \Delta_1(n,r) \subseteq\e+ S$ if and only if also $ \Delta_1(n-e_1,r) $ is an infinite line of $S$, and conversely $\Delta_1(n,r) \subseteq \Ap(S)$ if $ \Delta_1(n-e_1,r) $ is not an infinite line of $S$ for any $r \in S_2$. It follows that for every $n$ there exists a unique $m \equiv n$ mod $e_1$ such that $ \Delta_1(m,r) $ is an infinite line of $\Ap(S)$. With the same argument it can be shown that the analogous situation happens on $S_2$ and therefore there are $e_2$ infinite lines of $\Ap(S)$ of the form $\Delta_2(q, t_i)$ corresponding to some elements $t_1, \ldots, t_{e_2} \in S_2$.
Now, notice that, by Lemma \ref{basic-Ap}(7), if an infinite line is contained in $\Ap(S)$, then its elements must be contained eventually in a level $A_i$. Moreover, by Lemma \ref{basic-Ap}(8), $A_i$ cannot contain more than two infinite lines and, if it contains two of them, they must be one of the form $\Delta_1(s_1, r)$ and the other of the form $\Delta_2(q, s_2)$. 
Applying inductively the definition of the levels $A_i$, it follows that the levels $A_e, A_{e-1}, \dots, A_{e-e_j+1}$ contain the $e_j$ infinite lines contained in $\Ap(S)$ and corresponding to the elements of $S_j$. 
\end{proof}

We conclude this section by proving a generalization of Remark \ref{remnumerical} holding for good semigroups. We show that if $(R, \mathfrak{m}, k)$ is an analytically unramified one-dimensional local reduced ring, its quotient ring $R/(x)$, where $x$ is a nonzero element having minimal value in the value semigroup $S=v(R)$, can be generate as a $k$-vector space by a set of $e$ elements having values in all the different levels of the Ap\'{e}ry Set of $S$.

\begin{theorem}
\label{anelloquoziente}
Let $(R, \mathfrak{m}, k)$ be an analytically unramified one-dimensional local reduced ring, having value semigroup $S=v(R)$ and let $x \in R$ such that $v(x)= \e = \min(S \setminus \lbrace \boldsymbol{0} \rbrace)$. Let $ \Ap(S)=\bigcup_{i=1}^e A_i $ be the Ap\'{e}ry Set of $S$. 
It is possible to construct a chain $$\al_1 < \al_2 < \ldots < \al_e \in S$$ such that $\alpha_i \in A_i$ and, for every collection of $f_i \in R$ having $v(f_i)= \alpha_i $, $$ \dfrac{R}{(x)}= \langle \overline{f_1}, \overline{f_2}, \ldots, \overline{f_e} \rangle_k. $$
\end{theorem} 


\begin{proof}
By Remark \ref{multiplycity}, the dimension over $k$ of $R/(x)$ is $e$, hence we only need to show that, for $i=1, \ldots, e$, we can find elements $\al_i$ such that the correspondent $f_i$ are linearly independent over $k.$

We set $\al_1= \boldsymbol{0}$ and then we define the other elements $\al_i$ using the following procedure: in case $\al_i \ll \be$ for some $\be \in A_{i+1}$, we may simply set $\al_{i+1}:= \be$. Otherwise, if $\al_i= \be \wedge \de $ with $\be \in \Delta^S_1(\al_i) \cap A_{i+1}$ and $\de \in \Delta^S_2(\al_i) \cap A_{i+1}$, by Lemma \ref{altro-Ap}(2), we have $\Delta_h^S(\al_i) \subseteq \Ap(S)$ and then we set $\al_{i+1}:= \be$ if $h=2$ and $\al_{i+1}:= \de$ if $h=1$ (if $\Delta^S(\al_i) \subseteq \Ap(S)$ we can take indifferently one of them).

Now, taking $f_i \in R$ such that $v(f_{i})= \al_i \in \Ap(S)$, we clearly get by Proposition \ref{prop1} that $ \overline{f_{i}} $ is nonzero in $R/(x)$. Then we consider  $ v(\sum_{i=1}^{e} \lambda_if_i) $ for $\lambda_i \in k$.
 Let $j$ be the minimal index such that $\lambda_j \neq 0$. If $\al_j \ll \al_{j+1}$, we obtain $$ v(\sum_{i=1}^{e} \lambda_if_i) = v( \lambda_jf_j)= \al_j \in \Ap(S) $$ and therefore $ \sum_{i=1}^{e} \lambda_i\overline{f_{i}} $ is nonzero in $R/(x)$. 
 Otherwise, we may assume $\al_{j+1} \in \Delta^S_1(\al_j)$ and hence,
  our procedure used to define the $\al_i$ implies now that $ \Delta^S_2(\al_j) \subseteq \Ap(S) $. It follows that $$ v(\sum_{i=1}^{e} \lambda_if_i) \in \Delta^S_2(\al_j) \cup \lbrace \al_j \rbrace \subseteq \Ap(S) $$ and thus $ \sum_{i=1}^{e} \lambda_i\overline{f_{i}} $ is nonzero in $R/(x)$.
\end{proof}

\section{Symmetric good semigroups}

In this section we describe more properties of the Ap\'{e}ry Set of a good semigroup in the symmetric case. 

\begin{definition} 
\label{defsimmetrici}
 A good semigroup $S$ is \it symmetric \rm if, for every $\al \in \N^2,$ $\al \in S$ if and only if $ \Delta^S(\g - \al) = \emptyset. $
\end{definition}

Symmetry is an interesting concept because, in the value semigroup case, it is equivalent to the Gorensteiness of the associated ring. Indeed, an analytically unramified 
one-dimensional local reduced ring is Gorenstein if and only if its value semigroup is symmetric.
But more in general, a symmetric good semigroup has other nice properties, that we list in next proposition. Some of them have been already proved in \cite[Proposition 3.2]{apery}. An interesting fact that we are proving is that it is possible to know the number of absolute elements of a symmetric good semigroup only looking at one of its (numerical) projections.

\begin{remark}
	The projections of a symmetric good semigroup may fail to be symmetric numerical semigroups, as one can see for instance in Fig. \ref{semdual}.
\end{remark}

\begin{prop} 
\label{simmetria}
Let $S \subseteq S_1 \times S_2 \subseteq \N^2$ be a symmetric good semigroup, let $\e=(e_1, e_2)$, $\g = (\gamma_1,\gamma_2)$ and $ \Ap(S) $ be defined as previously. 
\begin{enumerate} 
 \item[(1)] If $ \al \in S $ is a absolute element, then also $\g - \al \in S$ and it is a absolute element.
 \item[(2)] The number of the absolute elements of $S$ is $n(S_1)- b(S_1)=n(S_2)- b(S_2),$ where $n(S_i)= |S_i \cap \lbrace 0, 1, \ldots, \gamma_i \rbrace|$ and $b(S_i)= |\N \setminus S_i|$.
\item[(3)] For $\al \in S$, $\al \in \Ap(S)$ if and only if $ \Delta^S(\g  + \e  - \al) \neq \emptyset. $
\item[(4)] If $\al \in \Ap(S)$, then $ \Delta^S(\g  + \e - \al) \subseteq \Ap(S)$.
 \item[(5)] Let $\al \in \N^2.$ If $ \Delta^S(\al) \subseteq \Ap(S)$ (possibly it is empty), then $ \g  + \e - \al \in S $. 
\item[(6)] Let $\al \in \Ap(S)$. Then for $i=1,2$; $\Delta^S_i(\g+ \e- \al) = \emptyset$ if and only if $\Delta^S_i(\al) \nsubseteq \Ap(S)$.
 \end{enumerate}
\end{prop}

\begin{proof}
\noindent (1) Follows by the definitions of symmetric semigroup and absolute element. 

\noindent (2) By definition of symmetric good semigroup, we have that $n \not \in S_1$ if and only if $\Delta^S_1(n,0) = \emptyset$, and if and only if $(\gamma_1-n, \gamma_2+m) \in S$ for every $m\geq 0$. It follows that the number of elements $s \in S_1$ such that $ \Delta_1(s_1,r) $  is an infinite line of $S$ (for some $r \in S_2$) is exactly $b(S_1)$. Call $M$ the number of absolute elements of $S$. Hence, $$\gamma_1 = M + 2b(S_1),$$ and, since $\gamma_1 = n(S_1)+ b(S_1)$, we obtain $M=n(S_1)- b(S_1).$ In the same way, we can show $M=n(S_2)- b(S_2).$

\noindent (3) An element $\al \in S$ is in $\Ap(S)$ if and only if $\al - \e \not \in S$, and this happens by Definition \ref{defsimmetrici} if and only if $ \Delta^S(\g  + \e  - \al) \neq \emptyset. $

\noindent (4) Since $\al \in \Ap(S)$, $ \Delta^S(\g - \al) = \emptyset. $ It follows that $ \Delta^S(\g + \e - \al)\subseteq \Ap(S). $

\noindent (5) If $ \Delta^S(\al) \subseteq \Ap(S)$, then by definition $ \Delta^S(\al - \e) = \emptyset.$ and therefore $\g  + \e - \al \in S$. 

\noindent (6) We prove the result for $i=1$. By (3) and (4), we have that $\Delta^S(\g+ \e- \al)$ is not empty and contained in $ \Ap(S) $. Assuming $\Delta^S_1(\g+ \e- \al) = \emptyset$, by axiom (G2) we get $ \g+ \e- \al \not \in S $ and $\Delta^S_2(\g+ \e- \al) \neq \emptyset$. By (5), it follows that $ \Delta^S(\al) \nsubseteq \Ap(S)$ and moreover, since there exists $\om \in \Delta^S_2(\g+ \e- \al) \subseteq \Ap(S)$, we get by (4) $ \Delta^S_2(\al) \subseteq \Delta^S_2(\g+ \e - \om) \subseteq \Ap(S) $. Hence $\Delta^S_1(\al) \nsubseteq \Ap(S)$.  

Conversely, if $\Delta^S_1(\al) \nsubseteq \Ap(S)$, there exists $\te \in \Delta^S_1(\al) \cap (\e+S)$ and therefore, again by (3) $ \Delta^S_1(\g + \e - \al) \subseteq \Delta^S_1(\g+ \e - \te) = \emptyset $.
\end{proof}


\bigskip

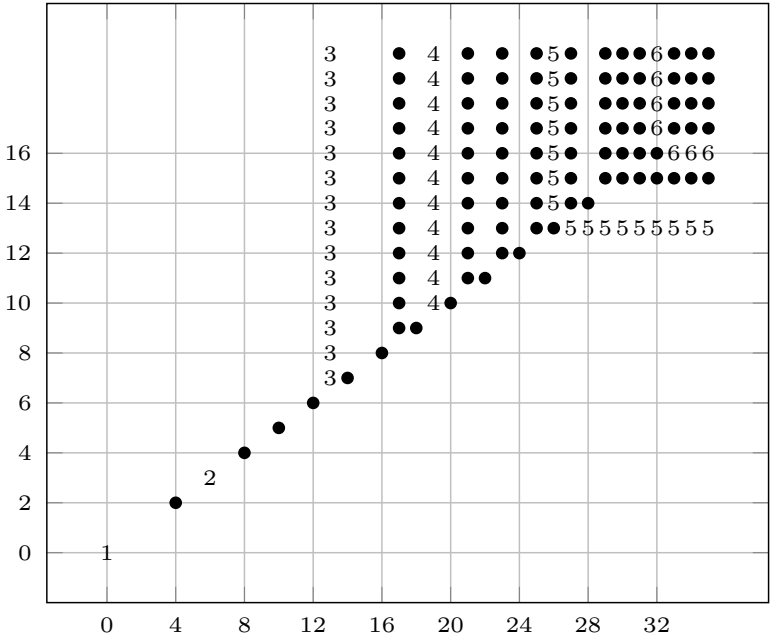
\begin{figure}[h]

 \begin{tikzpicture}[scale=1.4]
	\begin{axis}[grid=major, ytick={0,2,4,6,8, 10, 12, 14, 16}, xtick={0,4,8,12,16,20, 24, 28, 32}, yticklabel style={font=\tiny}, xticklabel style={font=\tiny}]

\addplot[only marks] coordinates{  (4,2) (8,4) (10,5) (12,6) (14,7) (16,8) (17,9) (18,9) (20,10) (22,11) (24,12) (26,13)
(29,15) (29,16) (29,17) (29,18) (29,19) (29,20)
(30,15) (30,16) (30,17) (30,18) (30,19) (30,20)
(31,15) (31,16) (31,17) (31,18) (31,19) (31,20) (32,16)
(32,15) (33,15) (34,15) (35,15) 
(33,17) (34,17) (35,17)
(33,18) (34,18) (35,18)
(33,19) (34,19) (35,19)
(33,20) (34,20) (35,20)
(28,14)
(17,10) (17,11) (17,12) (17,13) (17,14) (17,15) (17,16) (17,17) (17,18) (17,19) (17,20)
(21,11) (21,12) (21,13) (21,14) (21,15) (21,16) (21,17) (21,18) (21,19) (21,20)
 (23,12) (23,13) (23,14) (23,15) (23,16) (23,17) (23,18) (23,19) (23,20)
 (25,13) (25,14) (25,15) (25,16) (25,17) (25,18) (25,19) (25,20)
 (27,14) (27,15) (27,16) (27,17) (27,18) (27,19) (27,20)}; 
 \addplot[only marks, mark=text, mark options={scale=1,text mark={ \tiny 1}}, text mark as node=true] coordinates{(0,0)}; 
 \addplot[only marks, mark=text, mark options={scale=1,text mark={ \tiny 2}}, text mark as node=true] coordinates{(6,3)}; 
 \addplot[only marks, mark=text, mark options={scale=1,text mark={ \tiny 3}}, text mark as node=true] coordinates{(13,7) (13,8) (13,9) (13,10) (13,11) (13,12) (13,13) (13,14) (13,15) (13,16) (13,17) (13,18) (13,19) (13,20)}; 
 \addplot[only marks, mark=text, mark options={scale=1,text mark={ \tiny 6}}, text mark as node=true] coordinates{ (32,17) (32,18) (32,19) (32,20) (33,16) (34,16) (35,16)  }; 
 \addplot[only marks, mark=text, mark options={scale=1,text mark={ \tiny 4}}, text mark as node=true] coordinates{(19,10) (19,11) (19,12) (19,13) (19,14) (19,15) (19,16) (19,17) (19,18) (19,19) (19,20) }; 
 \addplot[only marks, mark=text, mark options={scale=1,text mark={ \tiny 5}}, text mark as node=true] coordinates{(26,14) (26,15) (26,16) (26,17) (26,18) (26,19) (26,20) (27,13) (28,13) (29,13) (30,13) (31,13) (32,13) (33,13) (34,13) (35,13)}; 
 \end{axis}
  \end{tikzpicture}

 \caption{\scriptsize The symmetric value semigroup of $k[[X,Y]]/(Y^4-2X^3Y^2-4X^5Y+X^6-X^7)(Y^2-X^3)$, see \cite[pag.8]{apery}. It is possible to observe that the number of absolute elements  of this semigroup is $13 = 21 - 8 = 14 - 1$ as predicted by the formula in Proposition \ref{simmetria}(2).
 }
\label{semlargo}
\end{figure}


Let $S \subseteq S_1 \times S_2 \subseteq \N^2$ be a good semigroup. We describe now, in the symmetric case, the absolute elements and the infinite lines of $\Ap(S)$ and of $\e+S$ in terms of the elements of a single projection, say $S_1.$
  For $n \in \N$, we consider the set $$ \Delta^S_1(n, 0)= \lbrace (n,m) \in S \, | \, m \geq 0 \rbrace. $$ This set can be empty, finite or infinite. It is infinite if and only if $\Delta^S_1(n, r)$ is an infinite line of $S$ for some $r \in S_2$; it is finite if and only if $ \Delta^S_1(n, 0) $ contains a absolute element of $S$; it is empty if and only if $n \not \in S_1.$
 The analogous situation holds for the other projection $S_2$.

\begin{lemma} 
\label{l1}
The set $ \Delta^S_1(n, 0) \subseteq \Ap(S) $ if and only if $n \in \Ap(S_1).$
\end{lemma}

\begin{proof}
We have $n \in \Ap(S_1)$ if and only if $n - e_1 \not \in S_1$ if and only if $ \Delta^S_1(n-e_1, 0) = \emptyset $. The result now follows from the definition of $\Ap(S).$
\end{proof}

\newpage

\begin{theorem} 
\label{proiezione}
Let $S \subseteq S_1 \times S_2 \subseteq \N^2$ be a symmetric good semigroup, let $\e=(e_1, e_2)$ be its minimal non-zero element. 
Let $n \in \N$ and define $n^{\prime}= \gamma_1 + e_1 - n.$
\begin{enumerate}
 \item[(1)] $ \Delta^S_1(n, 0) = \emptyset $ if and only if $ \Delta^S_1(n^{\prime}, 0) $ is infinite and eventually contained in $ \e + S. $
 \item[(2)] $ \Delta^S_1(n, 0) $ is finite with maximal element in $ \e + S $ if and only if $ \Delta^S_1(n^{\prime}, 0) $ is finite with maximal element in $ \e + S $.
 \item[(3)] $ \Delta^S_1(n, 0) \nsubseteq \Ap(S)$ and it is finite with maximal element in $ \Ap(S) $ if and only if $ \Delta^S_1(n^{\prime}, 0) \nsubseteq \Ap(S)$ and it is finite with maximal element in $ \Ap(S) $.
 \item[(4)] $ \Delta^S_1(n, 0) \subseteq \Ap(S) $ and it is finite if and only if $\Delta^S_1(n^{\prime}, 0)$ is infinite and eventually contained in $\Ap(S)$ but it contains some element of $ \e + S. $
 \item[(5)] $ \Delta^S_1(n, 0) \subseteq \Ap(S)$ and it is infinite if and only if $ \Delta^S_1(n^{\prime}, 0) \subseteq \Ap(S) $ and it is infinite.
 \end{enumerate}
 All the correspondent statements hold replacing $S_1$ with $S_2.$
\end{theorem}

\begin{proof}
\noindent (1) Observe that $ \Delta^S_1(n, 0) = \emptyset $ if and only if $ \Delta^S(n,-m) = \emptyset $ for all $m \geq 0$. This is equivalent by Definition \ref{defsimmetrici} to say that $(\gamma_1 - n, \gamma_2+ m) \in S$. Hence $ \Delta^S_1(n, 0) = \emptyset $ if and only if $ \Delta^S_1(\gamma_1 - n, 0) $ is infinite, which is equivalent to say that $  \Delta^S_1(\gamma_1 + e_1 - n, 0) =\Delta^S_1(n^{\prime}, 0) $ is infinite and contained in $ \e + S.$

\noindent (2) Let $\al \in \e+S$ be the maximal element of $S$ belonging to $ \Delta^S_1(n, 0) $. Hence $\al$ is an absolute element and, by Proposition \ref{simmetria}(1) $\g - \al $ is also an absolute element of $S$. It follows that $\al^{\prime} = \g + \e - \al \in \e+S$. Moreover, since $\al \in \e+S$, $ \Delta^S(\al^{\prime}) = \emptyset $ by Proposition \ref{simmetria}(3). Thus $ \al^{\prime} $ is a absolute element and $ \Delta^S_1(n^{\prime}, 0) $ is finite with maximal element in $ \e + S $. The converse is tautological.

\noindent (3) Let $ \Delta^S_1(n, 0) \nsubseteq \Ap(S)$ and it is finite with maximal element $\al= (n, m) \in \Ap(S) $. As in (2), we have that $\al^{\prime} = \g + \e - \al \in \e+S$. But, by exclusion, (1) and (2) imply that $ \Delta^S_1(n^{\prime}, 0) $ neither is infinite and eventually contained $ \e + S $ nor has a maximal in $ \e + S. $ Hence there must exist $\be = (n^{\prime}, r) \in \Ap(S)$ with $r > m^{\prime}:= \gamma_2 + e_2 -m $.
We conclude saying that, since $ \Delta^S_1(n, 0) \nsubseteq \Ap(S)$, there must exist an element $\de = (n, d) \in \e+S$ with $d < m$ and hence $\Delta^S(\de^{\prime})= \emptyset$. Thus there are only finite elements $ (n^{\prime}, q) \in S $ with $q > m^{\prime}$ and they are all in $ \Ap(S) $ by Proposition \ref{simmetria}(4), since they are elements of $ \Delta^S(\al^{\prime}) $. Since there exists at least one of such elements, namely $\be$, the thesis follows.

\noindent (4) %
Assume that $ \Delta^S_1(n, 0) \subseteq \Ap(S) $ and it is finite. Again by (1) and (2) we exclude that $ \Delta^S_1(n^{\prime}, 0) $ is infinite and eventually contained $ \e + S $ or has a maximal in $ \e + S. $ Let $\al= (n, m) \in \Ap(S) $ be the maximal element in $ \Delta^S_1(n, 0) $. We proceed like in the proof of (3) to say that $\al^{\prime} = \g + \e - \al \in \e+S$.
If by way of contradiction $ \Delta^S_1(n^{\prime}, 0) $ contains a maximal element $\te \in \Ap(S)$, it would follow by Proposition \ref{simmetria}(3 and 5) that $\te^{\prime} \in \Delta^S_1(n, 0) \cap (\e + S) $ and this is a contradiction.

Conversely, assume $ \Delta^S_1(n^{\prime}, 0) $ is infinite and eventually contained in $\Ap(S)$ but it contains some element $ \te \in \e + S. $ Since $ \Delta^S(\te^{\prime})= \emptyset $ (by Proposition \ref{simmetria}(3)), $ \Delta^S_1(n, 0) $ must be finite. We conclude by exclusion, since we characterized in (2) and (3) the other possible cases of a finite $ \Delta^S_1(n, 0) $.

\noindent (5) It follows since we excluded all the other possible cases in (1),(2),(3) and (4).
\end{proof}

\begin{corollary} 
\label{c1}
Assume the same notations of Theorem \ref{proiezione}. Hence, $n \in \Ap(S_1)$ if and only if $ \Delta^S_1(n^{\prime}, 0) $ is infinite and eventually contained in $\Ap(S)$.
\end{corollary}



\section{Duality of the Ap\'{e}ry Set of symmetric good semigroups}

The symmetry of a numerical semigroup $S$ can be characterized by the symmetry of its Ap\'ery Set with respect to its largest element: if we order the elements of $\Ap(S)$ in increasing order
$\Ap(S)=\{w_1, \dots, w_e\}$, then $S$ is symmetric if and only if $w_i+w_{e-i+1}=w_e$. 

Hence, there is a duality relation associating to each element $w_i$, the element $w_{e-i+1}$. In the case of a symmetric good semigroup we do not have this relation by choosing arbitrary elements, one from each level $A_i$ of $\Ap(S)$; but we find a more general duality relation associating the level $A_i$ to the level $A_{e-i+1}$, and involving both the elements of $\al \in A_i$ and the sets $\Delta^S(\g + \e - \al)$. After two preparatory lemmas, we define and prove this duality in Theorem \ref{duality}. 

In this Section, we denote as before $ \al^{\prime}:= \g + \e - \al. $

\begin{lemma} 
\label{du1}
Let $S \subseteq S_1 \times S_2 \subseteq \N^2$ be a symmetric good semigroup. Let $ \Ap(S)=\bigcup_{i=1}^e A_i $ be the Ap\'{e}ry Set of $S$. If $\al \in A_{e-i+1}$, then for every $j < i$, $$ \Delta^S(\al^{\prime}) \cap A_j = \emptyset. $$
\end{lemma}

\begin{proof}
We use induction on $i$. For $i=1$, the result is clear. 
Let $\al \in A_{e-i+1}$. We separate the proof in two cases: \\
\bf Case 1: \rm Assume $\al \ll \te $ for some $\te \in A_{e-i+2}$. By inductive hypothesis $ \Delta^S(\te^{\prime}) \cap A_j = \emptyset $ for every $j < i-1$. Since $\al \ll \te $, it follows that $\te^{\prime} \ll \al^{\prime} $ and hence for every $ \de \in \Delta^S(\al^{\prime}) $ there exists $ \be \in \Delta^S(\te^{\prime}) $ such that either $\be \ll \de$ or there exists $\om = \de \wedge \be \in \Delta^S(\te^{\prime}) \subseteq \Ap(S)$. In the first case the level of $\be$ in $\Ap(S)$ is smaller than the level of $\de$. In the second case, as a consequence of Lemma \ref{altro-Ap}(3), the element $ \om  $ is in a level of $\Ap(S)$ smaller than the level of $\de$. Hence we can find elements in $ \Delta^S(\te^{\prime}) $ in some level smaller than the level of any element of $ \Delta^S(\al^{\prime}) $. It follows that $ \Delta^S(\al^{\prime}) \subseteq \bigcup_{j \geq i} A_j $ and hence the thesis. \\
\bf Case 2: \rm Now assume $ \al = \te \wedge \de $ with $  \te \in \Delta^S_1(\al) \cap A_{e-i+2} $ and $ \de \in \Delta^S_2(\al) \cap A_{e-i+2}$. Hence $ \al^{\prime} \in \Delta_1(\te^{\prime}) \cap \Delta_2(\de^{\prime}). $ Assuming $ \Delta^S_1(\al^{\prime}) \neq \emptyset $ and taking $\om \in \Delta^S_1(\al^{\prime})$, we can find an element $ \be \in \Delta^S(\de^{\prime}) $ such that either $\be \ll \om$ or there exists $\be^1 = \om \wedge \be \in \Delta^S(\de^{\prime}) \subseteq \Ap(S)$ (it is possible to have $ \be^1= \al^{\prime} $). Using the same argument of Case 1, we show that one element among $\be$ and $\be^1$ is in a level of $\Ap(S)$ smaller than the level of $\om$ and therefore we get the same thesis of Case 1.
In case $ \Delta^S_2(\al^{\prime}) \neq \emptyset $ we can use the same argument to find the needed elements in $ \Delta^S(\te^{\prime}). $
\end{proof}

\begin{lemma} 
\label{du2}
Let $S \subseteq S_1 \times S_2 \subseteq \N^2$ be a symmetric good semigroup. Let $ \Ap(S)=\bigcup_{i=1}^e A_i $ be the Ap\'{e}ry Set of $S$. If $\al \in A_{i}$, then $$ \Delta^S(\al^{\prime}) \cap A_{e-i+1} \neq \emptyset. $$ 
\end{lemma}

\begin{proof}
Again we use induction on $i$. For $i=1$, the result follows since $A_1=\lbrace \boldsymbol{0} \rbrace$ and $\Delta^S(\boldsymbol{0}^{\prime})= \Delta^S(\g + \e) = A_e.$ 
Let $\al \in A_{i}$. By Lemma \ref{du1}, we have that $ \Delta^S(\al^{\prime}) \subseteq \bigcup_{j \geq e-i+1} A_j $. By Proposition \ref{caso3} we have that $\al \geq \te $ for some $\te \in A_{i-1}$. We separate the proof in two cases: \\
\bf Case 1: \rm Assume $\al \gg \te $. By inductive hypothesis, we know that there exists $ \be \in \Delta^S(\te^{\prime}) \cap A_{e-i+2} $. Using the argument of the proof of Lemma \ref{du1}(Case 1) we can show that there exists some element $ \om \in \Delta^S(\al^{\prime})$ which is in a level of $\Ap(S)$ smaller than $ A_{e-i+2}. $ Thus we must have $\om \in A_{e-i+1}.$ \\
\bf Case 2: \rm Now assume $\al \in \Delta^S(\te)$. Without loss of generality, we say that $\al \in \Delta^S_1(\te)$. Now, if $ \Delta^S_1(\te) \nsubseteq \Ap(S) $, by Proposition \ref{simmetria}(6) we have $ \Delta^S_1(\te^{\prime}) = \emptyset $ and therefore $ \Delta^S_2(\te^{\prime}) \neq \emptyset $. Otherwise, if $ \Delta^S_1(\te) \subseteq \Ap(S) $, applying Lemma \ref{altro-Ap}(5) we find an element $\om \in (\Delta^S_1(\te) \cap A_{i-1}) \cup \lbrace \te \rbrace$ such that $ \Delta^S_2(\om) \subseteq \Ap(S) $ and hence again by Proposition \ref{simmetria}(6), $ \Delta^S_2(\om^{\prime}) \neq \emptyset $. In both cases we have found an element $ \om \in A_{i-1} $ such that $\al \in \Delta^S_1(\om)$ and $ \Delta^S_2(\om^{\prime}) \neq \emptyset $. Proceeding like in Case 2 of Lemma \ref{du1} and using the inductive hypothesis, we find an element $ \be \in \Delta^S(\al^{\prime}) $ which is in a level of $\Ap(S)$ smaller than $ A_{e-i+2}. $ Thus we must have $\be \in A_{e-i+1}$ as in Case 1. 
\end{proof}

\begin{theorem} 
\label{duality}
Let $S \subseteq S_1 \times S_2 \subseteq \N^2$ be a good semigroup. Let $ \Ap(S)=\bigcup_{i=1}^e A_i $ be the Ap\'{e}ry Set of $S$. Denote $$ A_i^{\prime} = \left( \bigcup_{\om \in A_i} \Delta^S(\om^{\prime}) \right) \setminus \left( \bigcup_{\om \in A_j \mbox{, } j < i} \Delta^S(\om^{\prime}) \right).$$ The following assertions are equivalent:
\begin{enumerate}
\item $S$ is symmetric.
\item  $A_i^{\prime}= A_{e-i+1}$ for every $i=1, \ldots, e.$
\end{enumerate}
\end{theorem}

\begin{proof}
(1) $\to$ (2):
Assume $S$ to be symmetric and notice that in this case by Proposition \ref{simmetria}(4), $A_i^{\prime} \subseteq \Ap(S)$. As a consequence of the definition of the levels, $ A_e^{\prime}= A_1 $ and $ A_1^{\prime}= A_e $, thus we can assume by induction $ A_j^{\prime}= A_{e-j+1} $ for $j < i$ and $  A_j= A_{e-j+1}^{\prime} $ for $j > e-i+1$. 

We first show $A_i^{\prime} \subseteq A_{e-i+1}.$ Let $\de \in A_i^{\prime}$, hence $\de \in \Delta^S(\om^{\prime})$ for some $\om \in A_i$ and $ \de \not \in \bigcup_{\te \in A_j \mbox{, } j < i} \Delta^S(\te^{\prime}) $. Since $\om \in A_i$, by Lemma \ref{du1}, $\de \not \in A_j$ for $j < e-i+1.$ By way of contradiction assume $\de \in A_j$ for some $j > e-i+1$. Thus, by inductive hypothesis $  A_j= A_{e-j+1}^{\prime} $ and $ e-j+1 < i$. Hence $\de \in \Delta^S(\te^{\prime})$ for some $\te \in A_{e-j+1}$, but this is a contradiction since $\de \in A_i^{\prime}$ and $ e-j+1 < i$.

Now we show the other containment $A_{e-i+1} \subseteq A_i^{\prime}.$ Let $\om \in A_{e-i+1}$ and take $\de \in \Delta^S(\om^{\prime}) \cap A_i$ which does exist by Lemma \ref{du2}. Hence $\om \in \Delta^S(\de^{\prime})$. We need to prove that $ \om \not \in \Delta^S(\te^{\prime}) $ for every $\te \in A_j$ with $j<i.$ If by way of contradiction, we assume $ \om  \in \Delta^S(\te^{\prime}) $ for some $\te \in A_j$ with $j<i,$ we can take a minimal $j$ such that this happens, and hence by definition of $A_j^{\prime}$ and by inductive hypothesis, we get $\om \in A_j^{\prime}= A_{e-j+1}= A_h$ with $h> e-i+1$. But this is impossible since the levels of $\Ap(S)$ are disjoint.
With the same proof it is possible to show that $A_i= A_{e-i+1}^{\prime}$ and continue with the induction to prove (2).

\noindent (2) $\to$ (1): We argue by way of contradiction. Assuming that $S$ is not symmetric, we can find $\al \not \in S$ such that $ \Delta^S(\g - \al) = \emptyset. $ Since there exists a minimal $k \in \N$ such that $\al + k \e \in S$ and for every $k$, also $ \Delta^S(\g - \al - k\e) = \emptyset, $ we may assume, replacing $\al$ by $\al + k \e$, that $\al + \e \in \Ap(S).$
Assuming $\al + \e \in A_i$, we show that $\al + \e \not \in A_j^{\prime}$ for every $j$, and therefore $A_i \neq A_{e-i+1}^{\prime}.$ Indeed, $$ \emptyset= \Delta^S(\g - \al) = \Delta^S(\g + \e - (\al + \e)) $$ and, if $\al + \e \in \Delta^S(\be^{\prime})$ for some $\be \in \Ap(S)$, we would have $\be \in \Delta^S(\g + \e - (\al + \e))$ and this is a contradiction.
\end{proof}


\begin{corollary} 
\label{cordual}
Let $S \subseteq \N^2$ be a symmetric good semigroup and let $\al \in A_{e-i+1}$. The minimal elements of $\Delta^S(\al^{\prime})$ with respect to $\le$ are in $A_i.$
\end{corollary}

\begin{proof}
By Lemma \ref{du1}, for every $j < i$, $ \Delta^S(\al^{\prime}) \cap A_j = \emptyset, $ while by Lemma \ref{du2}, $ \Delta^S(\al^{\prime}) \cap A_i \neq \emptyset. $ Hence there exists a minimal element $\be$ of $ \Delta^S(\al^{\prime}) $ in $A_i$.
If $\te$ is another minimal element of $\Delta^S(\al^{\prime})$, we clearly have $\al^{\prime} = \be \wedge \te \in S$, and hence $\te \in A_i$ by Lemma \ref{altro-Ap}(3).
\end{proof}






In the next theorem, we provide a specific sequence of elements of a good semigroup $S$, taken one from each level $A_i$, behaving like the elements of the Ap\'{e}ry Set of a numerical semigroup with respect to sums. Notice that this sequence may not be the unique having the required property, but we give here a canonical way to construct one.

\begin{theorem} 
\label{sequenza2}
Let $S \subseteq S_1 \times S_2 \subseteq \N^2$ be a symmetric good semigroup and let $ \Ap(S)=\bigcup_{i=1}^e A_i $ be the Ap\'{e}ry Set of $S$. Assume $e_1 \geq e_2$.  
\begin{enumerate}
\item If $e$ is even, there exists a sequence of elements $ \al_1, \al_2, \ldots, \al_e $ such that $$ \al_i \in A_i $$ and $$ \al_i + \al_{e-i+1} = \al_e. $$  
\item If $e$ is odd, set $e=2d-1$. Then, there exists a sequence of elements $ \al_1, \al_2, \ldots, \al_e, \be $ such that $$ \al_i \in A_i, \, \be \in A_d$$ and $$ \al_i + \al_{e-i+1} = \al_e $$ for $i \neq d$, and moreover $$ \al_d + \be =  \al_e$$
\end{enumerate}
\end{theorem}

\begin{proof}
Let $\Ap(S_1)= \lbrace \omega_1=0, \omega_2, \ldots, \omega_{e_1} \rbrace$ be the Ap\'{e}ry Set of $S_1$ with elements listed in increasing order.
For $i=1, \ldots, e_1$ set $$\al_{i}:= \min \, \Delta^S_1(\omega_i, 0).$$ We observe that, defined in this way, $\al_i \in A_i$, since, by Corollary \ref{c1}, the set $\Delta^S_1(\omega_i^{\prime}, 0) $ is eventually contained in $\Ap(S)$ and in particular, by Theorem \ref{infiniti} it is eventually contained in $ A_{e-i+1} $. Hence, we get $\al_i \in A_i$ by Corollary \ref{cordual}.  Moreover, there exists a minimal $h_i \geq 0$ such that $ \g + \e -\al_i + (0, h_i) \in A_{e-i+1}. $

Call $H:= \max \lbrace h_i \rbrace$ and define again for $i=1, \ldots, e_1$, $$ \al_{e-i+1}:= \g + \e -\al_i + (0, H).$$ It follows that $ \al_{e-i+1} \in A_{e-i+1} $ and that $ \al_i + \al_{e-i+1} = \g + \e + (0, H)= \al_{e}. $ The second assertion is proved in the same way by defining $ \be:= \g + \e -\al_d + (0, H).$
\end{proof}

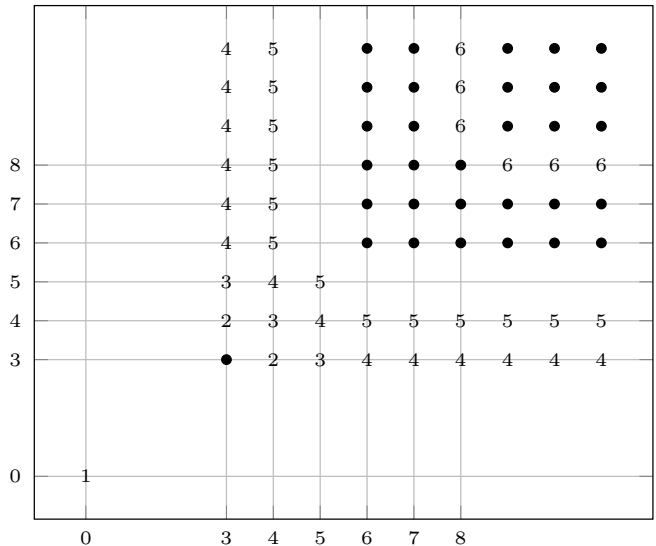
\begin{figure}[h]
 
 \begin{tikzpicture}[scale=1.2]
	\begin{axis}[grid=major, ytick={0,3,4,5,6,7,8}, xtick={0,3,4,5,6,7,8}, yticklabel style={font=\tiny}, xticklabel style={font=\tiny}]
\addplot[only marks] coordinates{ (3,3) (6,6) (6,7) (6,8) (7,6) (7,7) (7,8) (8,8) (8,7) (8,6) (6,9) (6,10)(6,11) (9,6) (10,6) (11,6) (7,9) (7,10)(7,11) (9,7) (10,7) (11,7) (9,9) (9,10) (9,11) (10,10) (10,9) (10,11) (11,9) (11,10)(11,11) }; 
 \addplot[only marks, mark=text, mark options={scale=1,text mark={ \tiny 1}}, text mark as node=true] coordinates{(0,0)}; 
 \addplot[only marks, mark=text, mark options={scale=1,text mark={ \tiny 2}}, text mark as node=true] coordinates{ (4,3) (3,4)}; 
 \addplot[only marks, mark=text, mark options={scale=1,text mark={ \tiny 3}}, text mark as node=true] coordinates{ (3,5) (4,4) (5,3)}; 
 \addplot[only marks, mark=text, mark options={scale=1,text mark={ \tiny 6}}, text mark as node=true] coordinates{ (8,9) (8,10) (8,11) (9,8) (10,8) (11,8)  }; 
 \addplot[only marks, mark=text, mark options={scale=1,text mark={ \tiny 4}}, text mark as node=true] coordinates{(3,6) (3,7) (3,8) (3,9) (3,10) (3,11) (6,3) (7,3) (8,3) (9,3) (10,3) (11,3) (5,4) (4,5) }; 
\addplot[only marks, mark=text, mark options={scale=1,text mark={ \tiny 5}}, text mark as node=true] coordinates{(4,6) (4,7) (4,8) (4,9) (4,10) (4,11) (6,4) (7,4) (8,4) (9,4) (10,4) (11,4) (5,5)};
 \end{axis}
  \end{tikzpicture}
  \caption{\scriptsize An example of a symmetric good semigroup whose projections are not symmetric numerical semigroups. It is a good example to check the duality property stated in Theorem \ref{duality}. 
  }
\label{semdual}
\end{figure}

We conclude giving a quite surprising result about symmetric good semigroup with large conductor.

\begin{prop} 
\label{largesemigroups}
Let $S \subseteq S_1 \times S_2 \subseteq \N^2$ be a symmetric good semigroup and let $ \Ap(S)=\bigcup_{i=1}^e A_i $ be the Ap\'{e}ry Set of $S$. Assume $e_1 \geq e_2$ and $$ \gamma_1 > 2f(S_1) + e_1 $$ where $f(S_1)$ denotes the Frobenius number of $S_1.$ Then, $e_1= e_2.$
\end{prop}

\begin{proof}
Set $\Ap(S_1)= \lbrace \omega_1=0, \omega_2, \ldots, \omega_{e_1} \rbrace$ with elements listed in increasing order. In the proof of Theorem \ref{sequenza2} is shown that $ \al_i= \min \, \Delta^S_1(\omega_i, 0) \in A_i $ and moreover by Theorem \ref{duality},
any element $ \te = (\gamma_1 + e_1 - \omega_i, t_2) $, with $t_2 \geq \gamma_2 + e_2$, is in $ A_{e-i+1} $ that is an infinite level by Corollary \ref{c1}. Since by assumption, $$\omega_{e_1}= f(S_1) + e_1 < \gamma_1 + e_1 - (f(S_1) + e_1) = \omega_{e_1}^{\prime},$$ we get $\al_{e_1} \ll \te \in A_{e-e_1+1}$ and therefore $e_1 < e-e_1+1 = e_2 + 1 $, since by Theorem \ref{livelli}, $e= e_1 + e_2.$ It follows that $e_1= e_2.$
\end{proof}

\end{document}